\documentclass[10pt]{amsart}
\usepackage[centertags]{amsmath}
\usepackage{amsfonts}
\usepackage{amsthm}
\usepackage{graphicx}
\usepackage{color}
\usepackage{bbding}
\usepackage[pdfstartview=FitH]{hyperref}
\usepackage[all]{xy}

\usepackage{amscd}
\usepackage{amssymb}
\usepackage{color}
\newlength{\defbaselineskip} 

\newtheorem{thm}{Theorem}[section]
\newtheorem{cor}[thm]{Corollary}
\newtheorem{lemm}[thm]{Lemma}
\newtheorem{prop}[thm]{Proposition}

\newtheorem{rem}[thm]{Remark}
\let\oldrem\rem
\renewcommand{\rem}{\oldrem\normalfont}

\theoremstyle{definition}

\newtheorem{remark}[thm]{Remark}

 \numberwithin{equation}{section}
\numberwithin{equation}{section} \theoremstyle{definition}
\DeclareMathOperator{\Pic}{Pic}
\DeclareMathOperator{\Aut}{Aut}
 \DeclareMathOperator{\Spec}{Spec}
 
\DeclareMathOperator{\Hom}{Hom}
\DeclareMathOperator{\Ext}{Ext}

\DeclareMathOperator{\id}{Id}

          \newcommand\PP{{\mathbb{P}}}
           
          \newcommand\QQ{{\mathbb{Q}}}
          \newcommand\ZZ{{\mathbb{Z}}}
           
          \newcommand\C{{\mathbb{C}  }}
          \newcommand\CC{{\mathbb{C}  }}
          
          \newcommand\R{{\mathcal{R}  }}
          \newcommand\s{{\mathcal{S}}}
          
           \newcommand\F{{\mathcal F}}
            
           \newcommand\G{{\mathcal G}}

          \newcommand\oo{\mathcal O}
         \newcommand\Z{\mathbb{Z}}

          \newcommand\Cl{\mathrm{Cl}}
          \newcommand\cone{\Pi}

\newcommand\lra{\longrightarrow}

\definecolor{zielony}{rgb}{0.5, 0.9, 0.1}
\definecolor{czerwony}{rgb}{0.8, 0.2, 0.1}
\definecolor{niebieski}{rgb}{0.3, 0.1, 0.9}

\newcounter{appendice}

\newcommand{\pb}{{\color{magenta}\bullet}}
\newcommand{\qua}{{\color{cyan}\blacksquare}\normalcolor}
\newcommand{\psp}{\bullet}
\newcommand{\pspuno}{{\color{red}\blacktriangledown}\normalcolor}
\newcommand{\pspdue}{{\color{green}\blacklozenge}\normalcolor}
\newcommand{\psptre}{{\color{blue}\hbox{\AsteriskCenterOpen}}\normalcolor}
\newcommand{\pspquat}{{\color{magenta}\hbox{\EightAsterisk}}\normalcolor}

\newcommand{\ignore}[1]{}
\newcommand{\bbeta}{\phi^{[2]}}
\newcommand{\aalpha}{\rho^{[2]}}
\begin{document}
\title{A very special EPW sextic and two IHS fourfolds} 
\author[M.~Donten-Bury]{Maria Donten-Bury}
\address{University of Warsaw, Poland and Freie Universit\"at Berlin, Germany}
\email{m.donten@mimuw.edu.pl}
\author[B.~van Geemen]{Bert van Geemen}
\address{University of Milan, Italy}
\email{lambertus.vangeemen@unimi.it}
\author[G.~Kapustka]{Grzegorz Kapustka}
\address{University of Z\"{u}rich, Switzerland, Polish Academy of Sciences and Jagiellonian University, Cracow, Poland}
\email{grzegorz.kapustka@uj.edu.pl}
\author[M.~Kapustka]{Micha\l{} Kapustka}
\address{University of  Stavanger, Norway and Jagiellonian University, Cracow, Poland}
\email{michal.kapustka@uis.no}
\author[J.~Wi\'{s}niewski]{Jaros\l{}aw A.~Wi\'{s}niewski}
\address{University of Warsaw, Poland}
\email{J.Wisniewski@mimuw.edu.pl}

\begin{abstract}

We show that the Hilbert scheme of two points on the Vinberg $K3$ surface has a 2:1 map onto a
very symmetric EPW sextic $Y$ in $\PP^5$. 
The fourfold $Y$ is singular along $60$ planes, $20$ of which form a complete
family of incident planes.
This solves a problem of Morin and O'Grady and establishes
that $20$ is the maximal cardinality of such a family of planes.
Next, we show that this Hilbert scheme is birationally isomorphic to
the Kummer type IHS fourfold $X_0$ constructed in \cite{DW}. 
We find that $X_0$ is also related to the Debarre-Varley abelian fourfold.
\end{abstract}

\maketitle

\section{Introduction}
By an irreducible holomorphic symplectic (IHS) fourfold we mean
a compact 4-dimensional, simply connected K\"{a}hler manifold with trivial canonical
bundle which admits a unique (up to a constant)
closed non-degenerate holomorphic $2$-form and is not a product of two
manifolds (see~\cite{B}). There are only two known families of such fourfolds:
the $21$-dimensional family of deformations of the Hilbert scheme of two points on a $K3$ surface 
(we say that elements of this family are \emph{of $K3^{[2]}$-type}) having $b_2=23$
and the $5$-dimensional family of deformations of the Hilbert scheme of three points summing to zero 
on an Abelian surface with $b_2=7$. 

A well known family of projective IHS fourfolds of $K3^{[2]}$-type is the family of double EPW sextics found by O'Grady in \cite{O}.
Recall that an EPW sextic is a special sextic hypersurface in $\PP^5$ constructed in 
\cite{EPW}. It arises from a choice of a subspace $A\subset \bigwedge^3 \C^6$,
Lagrangian with respect to the symplectic form $\eta$ 
(unique up to choice of volume form on $\bigwedge^6 \C^6$) on $\bigwedge^3 \C^6$ given by the wedge product. More precisely the EPW sextic associated to $A$ is 
$$
Y_A=\{[v]\in \mathbb{P}(\C^6)|\quad \dim( A\cap (v\wedge \textstyle \bigwedge^2 \C^6))\geq 1 \}.
$$
Such EPW sextics appear as quotients of polarised IHS fourfolds of $K3^{[2]}$-type by an anti-symplectic involution.

In this paper, we investigate birational models of a very special IHS fourfold of K3 type.  The fourfold is obtained as
a Hilbert scheme of two points on a special K3 surface studied by Vinberg. We find out, in particular, that on one hand the fourfold is birational to a double EPW sextic (see Proposition \ref{intro main epw}) and on the other it is birational to the Kummer type IHS fourfold constructed in 
\cite{DW} as a desingularisation of a quotient of an Abelian fourfold by a group action (see Theorem \ref{desingularisation}). 

To introduce our fourfold, let us denote by $S$ the $K3$ surface which is 
the desingularisation of the double covering of the Del Pezzo surface of degree 5, 
denoted by $S_5$, branched over the union of its 10 lines. 
This surface was studied by Vinberg in~\cite{vinberg} as one of two
``most algebraic K3 surfaces". The starting point of our investigation is the following proposition.
 
\begin{prop}\label{intro main epw}
There exists a series of flops from $S^{[2]}$ to a fourfold $\overline{S^{[2]}}$ and a generically $2:1$ morphism $\overline{S^{[2]}}\to Y$ to an EPW sextic
$Y\subset \PP^5$ which is singular along $60$ planes. 
\end{prop}

This EPW sextic $Y$ is a very symmetric and natural sextic  related to many classical objects in algebraic geometry. First, it is invariant with respect to the action of the symmetric group $\Sigma_6$ acting on $\PP^5$ by permutation of the coordinates. 
We prove in Section \ref{secEPW} that there are $16$ hyperplanes in $\PP^5$
which are tangent to $Y$ along Segre cubics 
(see \cite{D1} for the discussion about such cubics) 
such that the fifteen planes contained in each such cubic are singular planes on $Y$.
Moreover, $Y$ admits $16$ singular points of multiplicity $4$ whose tangent cone is 
the cone over the Igusa quartic. 
Furthermore, $Y\subset \PP^5$ is projectively self-dual. 
Finally, the EPW sextic $Y$ gives rise to a maximal configuration of 
$20$ incident planes in $\mathbb{P}^5$, 
thus solving a classical problem proposed by Morin and reformulated by O'Grady. 

The problem of O'Grady addresses the question of finding \emph{complete families of incident planes} in $\PP^5$
i.e. configuration of planes in $\PP^5$ intersecting each other and such that no 
planes outside of this set  intersects all of the planes in the set. 
In 1930 Ugo Morin classified in \cite{Morin} all complete irreducible families of incident planes in $\PP^5$. 
In the same paper he acknowledged that the classification of complete finite set of incident planes 
presents essential difficulties.
The latter problem was re-addressed by Dolgachev and Markushevich \cite{DM} who found, 
using the geometry of the Fano model of an Enriques surfaces, 
a description of some families of $10$ incident planes. 
Moreover, they found an explicit description of a complete family of $13$ incident planes.
Then O'Grady in \cite{O} proved, using the results of Ferretti, that for $10\leq k \leq 16$ 
there exists a $20-k$ dimensional moduli space of complete families of incident planes of cardinality $k$. 
Moreover, he proved that the maximal cardinality of a finite complete family of planes 
is between $10$ and $20$. Then O'Grady asked: what is the maximal cardinality of a finite family of 
incident planes in $\PP^5$?

Finding families of incident planes is in fact strictly related to EPW sextics. 
Indeed from \cite[Claim 3.2]{ogrady-incident}, the set of points in $G(3,\C^6)$ corresponding to any complete family 
$\mathcal{F}$ of incident planes spans a space 
$\mathbb{P}(A_{\mathcal{F}})\subset \mathbb{P}(\bigwedge^3 \C^6)$
with $A_{\mathcal{F}}$ Lagrangian with respect to $\eta$. 
In particular, a finite complete family of incident planes $\mathcal{F}$ gives 
rise to an EPW sextic $Y_{A_{\mathcal{F}}}$. 
The configuration of planes is then contained in the singular locus of the EPW sextic. It follows that the construction of a finite complete family of incident planes amounts to finding a special EPW sextic. In our case, it turns out that a suitable subset of $20$ of the $60$ singular planes of the EPW sextic $Y$
is a complete family of incident planes of cardinality 20. Note that the configuration of this $20$ planes is probably rigid. 
We hence infer the following answer to O'Grady's problem:
\begin{thm}\label{20 incident}
There exists a complete family of twenty incident planes in $\PP^5$. 
So $20$ is the maximal possible cardinality of a finite complete family of incident planes in $\PP^5$.
\end{thm}

The outline of the first part of the paper is as follows. In Section~\ref{section k3s} we construct a rational map $S^{[2]}\rightarrow \PP^5$ and find its image. The flops of Proposition 
\ref{intro main epw} are described in Section~\ref{section map}, see Proposition~\ref{propMain}.
Then, the detailed proof of the Theorem \ref{20 incident} is given in section \ref{section epw sextic}. In Section \ref{secEPW} further geometric properties of the EPW sextic are studied.


The second aim of the paper is to establish a relation between $S^{[2]}$ and the
Kummer type IHS fourfold constructed in \cite{DW} as a desingularisation of an 
Abelian fourfold by a group action. 
We shall see that these manifolds are in fact related through the sextic $Y$.

Let us recall the construction from \cite{DW}.
Bellamy and Schedler in \cite{BS} showed that a certain action of the group 
$G:=Q_8\times_{\Z_2}D_8$, of order $32$,
on $\C^4$ admits a symplectic resolution.
This action can be given by matrices with coefficients in $\Z[i]$. 
Thus we get an action of $G$ on $E^4$, where $E$ is the elliptic curve
$E:=\C /(\Z+i\Z)$ with complex multiplication by $\Z[i]$.
Donten-Bury and Wi\'sniewski \cite[Sect.~6]{DW} showed that there is an inclusion 
$G\subset \Aut(E^4)$ such that the quotient
$E^4/G$ admits a symplectic resolution $X_0$ (not uniquely defined)
which is an IHS fourfold with second Betti number
$b_2(X_0)=23$. One of the goals of our paper is to show that $X_0$ is of $K3^{[2]}$-type. We shall see that it is in fact birational to the Hilbert scheme of two points on a $K3$-surface.

To relate $E^4/G$ to $S^{[2]}$ and $Y$, 
we find a morphism $E^4\rightarrow \PP^5$ which factors over the group 
$(G,i)\subset Aut(E^4)$
generated by $G$ and diagonal scalar multiplication by $i$. 
This morphism is proven to be the 
quotient map $E^4\rightarrow E^4/(G,i)$ and $E^4/(G,i)$ is shown to be a sextic hypersurface $Y'$ in 
$\PP^5$. After finding special geometrical properties of $Y'$ (singular planes, tangent hyperplanes etc.)
we are finally able in Section \ref{secDesing} to show that actually $Y=Y'$. 
We then conclude that both $\overline{S^{[2]}}$ and $E^4/G$ admit a 2:1 morphism to the same sextic $Y$ 
and have the same ramification locus.
We also obtain:

\begin{thm} \label{desingularisation} There exists a birational contraction 
$\overline{S^{[2]}}\to E^4/G$. In particular $X_0$ is deformation equivalent to the Hilbert scheme of two points on a $K3$ surface. 
\end{thm}
In Remark (\ref{remEx}) we describe explicitly the exceptional locus
of the contraction $\overline{S^{[2]}}\to E^4/G$ over the quotient
singularity of type $\C^4/G$ which was studied in \cite{BS}
(cf.~\cite{DW}).  
 
To obtain the map from $E^4$ to $\PP^5$, we show first in section \ref{secPolarization} 
that there is a unique $G$-invariant principal polarization $H$ on $E^4$.
This is not the product polarization,
but $(E^4,H)$ is a principally polarized Abelian fourfold with the maximal 
number of $10$ vanishing theta nulls found in \cite{V} and \cite{D}.
The singular locus of any theta divisor of $(E^4,H)$
consists of exactly 10 ODPs.
We study in Section \ref{aute4} the automorphism group of $(E^4,H)$ and we deduce
that the configuration of the sixteen
$G$-fixed points in $E^4$ and the singular points on the sixteen $G$-invariant theta divisors
is a $(16,6)$ (actually a complementary $(16,10)$) configuration. The map $E^4\to Y$ is given by a subsystem of $|2D|$ where $D$ is a $G$-invariant theta divisor (see Remark \ref{diagram qoutient}).
 
It is well-known that the K3 surface $S$ is a desingularization of a quotient of $E^2$ by a diagonal action of $\mathbb{Z}_4$ generated by $(i,-i)$, 
 so it is maybe not so surprising that
$E^4/G$ and $S^{[2]}$ are closely related (see also \cite{vgeementop} for another description of $S$). In fact, there is a lattice theoretical argument proving that
(see Proposition \ref{mongardi}). However, the groups involved in the quotient maps 
$(E^2)^2\rightarrow S^{[2]}$
and in $E^4\rightarrow E^4/G$ are different and we do not know of a 
more direct method to obtain Theorem \ref{desingularisation}. In any case, the sextic $Y$ is clearly 
 of independent interest.

\noindent{\bf Notation.}\par
\par
\begin{tabular}{ll}
  $\Sigma_n$& $n$-th permutation group\\
  $S_5\subset\PP^5$&degree 5 Del Pezzo surface via anticanonical embedding\\
  $\PP^2_r$& blow-up of $\PP^2$ at $r$ general points, $S_5\cong\PP^2_4$\\
  $\alpha:\Sigma_5\rightarrow Aut(S_5)$&action of $\Sigma_5$ on $S_5$, second wedge of standard representation\\
  $S'=\Pi\cap Q\subset\PP^6$& intersection of the cone $\Pi$ over $S_5$ with vertex $P$ and a special quadric $Q$\\
  $\rho: \PP^6\supset S'\rightarrow S_5\subset\PP^5$&projection from the vertex of $\Pi$\\
  $C=Q\cap S_5\subset S_5$& ramification of $\rho$, union of ten $(-1)$ curves on $S_5$\\
  $\phi:S\rightarrow S'$& resolution of 15 $A_1$ singularities, $S$ is a Vinberg K3 surface\\
  $\nu: S\rightarrow S_5$& composition of the above two morphisms \\
  $e_1,\dots,e_{15}$&$(-2)$-curves on $S$ which are contracted by $\phi$\\
  $l_1,\dots,l_{10}$&$(-2)$-curves on $S$ in strict transform of branching of $\rho$\\ 
  $S^{[2]}=Hilb^2(S)$&Hilbert scheme of sub-schemes of $S$ of length 2\\
  $\mu: DivS\rightarrow DivS^{[2]}$&$\mu(C)$ consist of cycles having non-empty intersection with $C\subset S$\\ 
  $\bbeta: S^{[2]}\dashrightarrow S'^{[2]}$&push-forward map \\
  $\aalpha: S'^{[2]}\dashrightarrow Y\subset\PP^5$&maps a pair of points on $S'$ to intersection of their span with $\PP^5$ \\
  $g: S^{[2]}\dashrightarrow \PP^5$&composition $\bbeta\circ\aalpha$\\ 
  $Y\subset \PP^5$& image of $g$, hypersurface of degree 6, EPW sextic\\
  $\beta:\Sigma_5\rightarrow Aut(\PP^5)$& action compatible with $\alpha$ and $g$\\
  $\overline{g}: \overline{S^{[2]}}\rightarrow Y\subset \PP^5$& small modification of $g$ which is a regular morphism\\
  $F_{ij}\subset S^{[2]}$& $F_{ij}=\{ \{p,q\}\in S^{[2]}:p\in e_i,\ q\in e_j\}\cong \PP^1\times\PP^1$\\
  $E_{ij}\subset S^{[2]}$& $E_{ij}=\{\{p,q\}\in S^{[2]}:p\in e_i,\ q\in l_j\}\cong \PP^2_2$\\
  $L_{ij}\subset\overline{S^{[2]}}$&strict transform of $L_{ij}'=\{\{p,q\}\in S^{[2]}\colon p\in l_i, q\in l_j \}\subset S^{[2]}$\\
  $\overline{g}\colon \overline{S^{[2]}}\xrightarrow{c} Z \xrightarrow{f} Y$&Stein factorization\\
\end{tabular}

\subsection*{Acknowledgments}
The following authors were supported by research grants from Polish
National Science Center: MD-B by project 2013/08/A/ST1/00804, GK by the project 2013/08/A/ST1/00312, MK by the project
2013/10/E/ST1/00688, JW by project 2012/07/B/ST1/03343.  The work was
completed while MD-B held a Dahlem Research School
Postdoctoral Fellowship at Freie Universit\"at Berlin.

We would like to thank Giovanni Mongardi for useful discussions.
\section{The Hilbert scheme of two points on the Vinberg K3 surface}\label{section_hilbert}
In this section we study the birational geometry of $S^{[2]}:=Hilb^2(S)$,
where $S$ is the (unique) K3 surface with transcendental lattice $T_S$
isomorphic to the the rank two lattice $(\ZZ^2,q=2(x^2+y^2))$, studied by Vinberg in~\cite{vinberg} (see also \cite{vgeementop}).  
Moreover, we exhibit an explicit (rational) 2:1 map
$g\colon S^{[2]}\dashrightarrow\PP^5$ whose image is an EPW sextic $Y$ and 
we find the equation defining $Y$.

\subsection{The K3 surface $S$ }\label{section k3s} 
The surface~$S$ is the desingularisation of the double
cover of $\PP^2$ branched along the union of six lines, 
an equation for the branch curve is:
\begin{equation}\label{sixlines}
xyz(x-y)(x-z)(y-z)\,=\,0~.
\end{equation}
The six lines defined by the equation (\ref{sixlines}) meet three at a time in 
the points
$$
p_1:=(1:0:0), \quad p_2:=(0:1:0),\quad p_3:=(0:0:1),\quad p_4:=(1:1:1)~,
$$
and two at a time in the points
$$(0:1:1),\quad (1:0:1), \quad (1:1:0)~.$$
Blowing up the four triple points $p_1,\dots, p_4$ in $\PP^2$ we obtain a Del Pezzo surface $S_5$. 
The strict transform of the 6 lines, together with the four exceptional divisors of the birational map $S_5\to \PP^2$, 
are the ten $(-1)$ curves on $S_5$. 
The reduced divisor $B$ in the Del Pezzo surface $S_5$, whose irreducible components are 
the ten $(-1)$ curves  has only ordinary double points, has arithmetic genus $6$ and 
lies in $|-2K_{S_5}|$. 
Hence, $B$
is even in the Picard group of $S_5$, 
so it is the branch locus of a degree two map $\rho\colon S'\rightarrow S_5$.

The linear system $|-2K_{S_5}|$ embeds $S_5$ as an intersection of $5$ 
quadrics $q_i'=0$, $i=1,\ldots,5$, in $\PP^5$. 
There is a further quadric $q_0'=0$ in $\PP^5$ which cuts out the divisor $B$ on $S_5$.
The K3 surface $S$ thus has a (nodal) model $S'$ as the intersection of the cone $\cone$ over $S_5$
with vertex $P=(0:\ldots:0:1)\in\PP^6$ (so $\cone$ is defined by the 5 quadrics $q_i'=0$ in $\PP^6$)
and the quadric $Q$ defined by $y_6^2=q_0'$, see the proof of \ref{equationY} below; 
that is we have a contraction map
$$
\phi\colon S\,\longrightarrow\,S'\,:=\,\cone\cap Q\quad(\subset\PP^6).
$$
The degree two map $\rho\colon S'\rightarrow S_5$ is simply the projection from $P$ to the hyperplane $y_6=0$.

The divisor $B$ in $S_5$ has $4\cdot 3+3=15$ ordinary double points and after blowing up these points 
we get a surface $S_5'$ on which the strict transforms of the $6$ lines and $4$ exceptional curves 
of the first blow up are disjoint. 
The $K3$ surface  $S$ is the double cover of $S_5'$ branched over these $10$ curves,
the inverse images of the $15$ rational curves over double points of $B$ are rational curves in $S$.
The map $\phi:S\rightarrow S'$ is the contraction of these $15$ rational curves. 
The intersection diagram of the $25$ curves on $S$ above is the following Petersen graph.
Here the dots ${\pb}$ correspond to the $10$ curves from $\mathcal{T}$
and the edges to the $15$ exceptional curves of $\phi \colon S\to S'$.

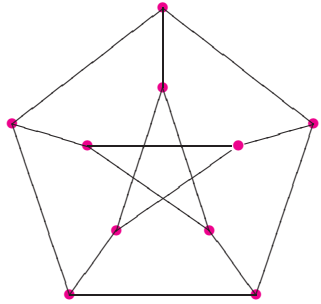
$$\begin{xy}<60pt,0pt>:
(0,1)*={\pb}="01", 
(0.588,-0.809)*={\pb}="02", 
(0.475,0.13)*{\pb}="12", 
(-0.588,-0.809)*={\pb}="14", 
(-0.475,0.13)*={\pb}="04", 
(-0.95,0.27)*={\pb}="23",  
(0.294,-0.404)*={\pb}="13", 
(0.95,0.27)*={\pb}="34", 
(-0.294,-0.404)*={\pb}="03",  
(0,0.5)*={\pb}="24",  
"01";"24" **@{-}, "01";"23" **@{-}, "01";"34" **@{-},
"02";"34" **@{-}, "02";"13" **@{-}, "02";"14" **@{-},
"03";"12" **@{-}, "03";"14" **@{-}, "03";"24" **@{-},
"04";"12" **@{-}, "04";"13" **@{-}, "04";"23" **@{-},
"12";"34" **@{-}, "13";"24" **@{-}, "14";"23" **@{-},
\end{xy}$$ 
\medskip

Let $C$ be the pull-back of $B$ to $S$ along the composition
$S\rightarrow S'\rightarrow S_5$. Then $C^2=10$ and the linear system
given by this curve is the contraction map $\phi=\phi_C:S\rightarrow
S'\subset\PP^6$.

\subsection{The map $g\colon S^{[2]}\dashrightarrow \PP^5$} \label{eq1}
Let $S^{[2]}$ be the Hilbert scheme of two points on $S$.
We define a rational map $g$ (cf.~\cite[Sect.~4.3]{O2}) 
that on an open part of $S^{[2]}$ is a composition of rational maps:
\[ 
g\colon S^{[2]}\stackrel{\bbeta}{\dashrightarrow} (S')^{[2]} 
\stackrel{\aalpha}{\dashrightarrow} Y\subset \PP^5~.
\]
Here the map $\bbeta$ is the map naturally induced by 
$\phi$ on the Hilbert scheme.
By $(S')^{[2]}$ we denote the
Hilbert scheme of two points in the smooth part of $S'$. 
The rational map $\aalpha$  
is induced by mapping $\{p,q\} \in
(S')^{2}$ to the hyperplane of those quadrics in the ideal of $S'\subset \PP^5$
which contain the line spanned by $p$ and $q$. 
The base locus of the map $g$ will be studied in Section \ref{section map}.
We shall see in Proposition \ref{propMain} that the map $g$ is given by a complete linear system.



To describe the image $Y$ of $g$, 
we need the following symmetric functions in six variables $Z_0,\ldots,Z_5$.
{\renewcommand{\arraystretch}{1.3}
$$
\begin{array}{l}
P_6:=Z_0^6 + Z_1^6 + Z_2^6 + Z_3^6 + Z_4^6 + Z_5^6,\\
P_{42}:=
Z_0^4  Z_1^2 + Z_0^4  Z_2^2+ \ldots + Z_4^2 Z_5^4,\\
P_{222}:=
Z_0^2  Z_1^2  Z_2^2 + Z_0^2  Z_1^2 Z_3^2 + \ldots + Z_3^2  Z_4^2  Z_5^2, 
\\
P_{111111}:=Z_0 Z_1 Z_2 Z_3 Z_4 Z_5,
\end{array}
$$
}
the polynomials $P_{42}$ and $P_{222}$ have 30 and 20 terms respectively.

\begin{prop}\label{equationY}
The image $Y$ of $S^{[2]}$ under $g\colon S^{[2]}\dashrightarrow\PP^5$ is the sextic $F_6=0$ where, 
with the notation as above,
$$
F_6\,=\,P_6 - P_{42} + 2P_{222} -16P_{111111}.
$$
\end{prop}

\begin{proof}
The anticanonical map from $S_5$ to $\PP^5$ is given by the linear system of cubics in $\PP^2$ which pass through the points $p_1,\ldots,p_4$. 
A basis for these cubics is given by:
{\renewcommand{\arraystretch}{1.3}
$$
\begin{array}{lll}
y_0\,=\,x_0^2x_1-x_0x_1x_2,\quad&
y_1\,=\,x_0^2x_2-x_0x_1x_2,\quad&
y_2\,=\,x_0x_1^2-x_0x_1x_2,\quad\\
y_3\,=\,x_0x_2^2-x_0x_1x_2,\quad&
y_4\,=\,x_1^2x_2-x_0x_1x_2,\quad&
y_5\,=\,x_1x_2^2-x_0x_1x_2.
\end{array}
$$
}
The image of $S_5$ in $\PP^5$ is defined by the following five quadratic forms:
{\renewcommand{\arraystretch}{1.3}
$$
\begin{array}{ll}
q_1'\,=\,y_0y_3 + y_1y_2 - y_2y_5 - y_3y_4,\quad & \\
q_2'\,=\,y_0y_4 + y_1y_2 - y_1y_5 - y_3y_4,\quad & \\
q_3'\,=\,   y_0y_5 + y_1y_2 - y_1y_5 - y_2y_5 - y_3y_4, \quad &  \\
q_4'\,=\,  y_1y_4 - y_1y_5 - y_3y_4,& \\  
q_5'\,=\,  y_2y_3 - y_2y_5 - y_3y_4.
\end{array}
$$
}
Moreover, the quadratic form
$$
q_0'\,:=\,y_1y_2- y_3y_4
$$
cuts out the union of the ten lines on $S_5$.
Thus the image $S'$ of the $K3$ surface $S$ in $\PP^6$ under the map $\phi$ is  
defined by the five quadrics which cut out the image of $S_5$ and the quadric
with equation $y_6^2=q_0'$, where now $y_0,\ldots,y_5,y_6$ are the homogeneous coordinates on $\PP^6$.
We define $q_6':=q_0'-y_6^2$. 
In order to get the very symmetric equation $F_6$ we need the following change of 
basis on the space of these quadrics:
{\renewcommand{\arraystretch}{1.3}
$$
\left(\begin{array}{c}
q_1\\ q_2\\ \\ \vdots\\ \\q_6\end{array}\right)\;=\;
\left(\begin{array}{rrrrrr}
1&1&0&0&-2&-1\\
-1&1&0&0&2&-1\\
1&-1&0&2&0&-1\\
-1&-1&2&0&0&-1\\
1&1&-2&0&0&-1\\
1&1&0&-2&0&-1
\end{array}\right)\,
\left(\begin{array}{c}
q_1'\\ q_2'\\ \\ \vdots\\ \\q_6'\end{array}\right).
$$
}

The map $g:S^{[2]}\dashrightarrow\PP^5$ maps $\{p,q\}\in
S^{[2]}$ to the hyperplane of quadrics in the ideal of $\phi(S)$
which contains the line spanned by $\phi(p)$ and $\phi(q)$.
If $\phi(p)=(y_0:\ldots:y_6)$ and $\phi(q)=(z_0:\ldots:z_6)$ then we can write
$$
q_i(\lambda y +\mu z)\,=\,\lambda^2q_i(y)\,+\,\mu^2q_i(z)\,+\,2\lambda\mu Q_i(y,z),
$$
where $Q_i$ is the symmetric bilinear form given by the polarization of $q_i$. 
The map $g$ is thus induced by the rational map
$$
\tilde{g}:\,S\times S\,\longrightarrow\,\PP^5,\qquad
\{p,q\}\,\longmapsto\,(\ldots:Q_i(\phi(p),\phi(q)):\ldots)_{1\leq i\leq 6}.
$$
It is now easy to verify that the polynomial $F_6$
vanishes on the image of $g$, one only needs to check that
$F_6(\ldots,Q_i(y,z),\ldots)=0$ on $S\times S$ where $y,z\in
\phi(S)$, so one substitutes the cubic polynomials in $x_0,x_1,x_2$
for $y_0,\ldots,y_5$, similarly for $z_0,\ldots,z_5$, but now with
polynomials in $u_0,u_1,u_2$ (coordinates for another copy of $\PP^2$)
and one uses that $y_6^2=y_1y_2- y_3y_4$, $z_6^2=z_1z_2- z_3z_4$.

For later use we notice that a general line contained in the cone $\cone$ over $S_5$ 
(defined by $q_1'=\ldots = q_5'=0$)
and passing through its vertex $P$ cuts $S'$ in two points $\phi(p),\phi(q)$.
Using the change of basis, we find that $\{p,q\}$ maps
to the point $(-1:\ldots:-1)=(1:\ldots:1)$ in $Y$.
\end{proof}

\begin{rem} Note that the equation of the image of $g$ can be found in a theoretical 
way using the results from 
section \ref{section map} and from section \ref{secDesing}. 
\end{rem}

\subsection{The sextic $Y$ is a special EPW sextic}\label{section epw sextic} 
We will show that the degree six fourfold $Y\subset\PP^5$, 
which is the image of $S^{[2]}$, is an EPW sextic (see \cite{O1}). 
The singular locus of a general EPW sextic is a surface of degree $40$. 
The sextic
$Y$ is (very) special in the sense that its singular locus has degree $60$, 
in fact it is the union of $60$ planes. 
The double cover of an EPW sextic along the singular locus is an IHS fourfold. 
Forty of the $60$ singular planes in $Y$ are in the branch locus of the map 
$g:S^{[2]}\rightarrow Y$, 
the other `extra' $20$ planes must then be a set of incident planes. 
To identify the planes in the branch locus, 
we use that the symmetric group $\Sigma_5$ acts on the Del Pezzo surface $S_5$ 
and that this action lifts to the action of a group $\widetilde{\Sigma}_5$ 
on the K3 surface~$S$. 
This group then also acts on $S^{[2]}$ and we show that $g$ is an equivariant map, 
where $\widetilde{\Sigma}_5$ acts through a subgroup $\beta\Sigma_5$, 
isomorphic to $\Sigma_5$, of $\Sigma_6$ on $Y$. 
Knowing the 20 incident planes then allows us to find explicitly 
a Lagrangian subspace 
$A\subset\bigwedge^3\C^6$ such that $Y=Y_A$ in the EPW construction. 

The first part of the following lemma is well-known.

\begin{lemm}\label{auts5}
  The symmetric group $\Sigma_5$ acts as group of automorphisms on
  $S_5$, the permutations of the points $p_1,\ldots,p_4$ induce the
  elements in $\Sigma_4\subset\Sigma_5$ and the Cremona transformation
  on $p_1,p_2,p_3$ induces the transposition $(45)$.

  These automorphisms of $S_5\subset\PP^5$, where the embedding is
  given by the cubics from the proof of Proposition \ref{equationY},
  are induced by projective transformations which map
  $y:=(y_0:y_1:\ldots:y_5)$ to
$$
\alpha_{34}(y):=(y_0-y_1+y_4:-y_1+y_3-y_5:y_2-y_4+y_5:-y_5:-y_3-y_4+y_5:-y_3).
$$
and the maps $\alpha_{12},\,\alpha_{23},\,\alpha_{45}$ permute the coordinates $y_i$ as
$$
\alpha_{12}:\;(02)(14)(35),\quad \alpha_{23}:\;(01)(23)(45),\quad 
\alpha_{45}:\;(05)(14)(23).
$$
The map $g:S^{[2]}\rightarrow \PP^5$ is equivariant for the action of
$\Sigma_5$, where the action of $\sigma\in\Sigma_5$ on $\PP^5$ is
given by the permutation $\beta_\sigma$ of the projective coordinates
$Z_0,\ldots,Z_5$:
$$
\beta_{12}:\;(03)(14)(25),\quad 
\beta_{23}:\;(01)(24)(35),\quad 
\beta_{34}:\;(05)(14)(23),\quad
\beta_{45}:\;(01)(25)(34).
$$
\end{lemm}

\begin{proof} These are straightforward verifications. The permutation
  of the points $p_3,p_4$ is given by $(x:y:z)\mapsto(x-z:y-z:-z)$ and
  now one computes the action on the six cubics from the proof of
  Proposition \ref{equationY}.  The Cremona transformation is induced
  by $(x:y:z)\mapsto (x^{-1}:y^{-1}:z^{-1})$, this is substituted in
  the cubics and next one multiplies them by $(xyz)^2$.

  Interchanging for example $x$ and $y$, the equation of the branch
  curve (\ref{sixlines}) changes sign, and the same happens for any
  other transposition in $\Sigma_5$.  Thus to lift the action to
  $S'\subset\PP^6$ one must map $y_6\mapsto iy_6$ (with $i^2=-1$).
  Finally, one considers the induced action on the quadratic forms
  $q_1,\ldots,q_6$ in the variables $y_0,\ldots,y_6$, these $q_i$ are the
  coordinate functions $Z_{i-1}$.
\end{proof}

To describe the singular locus of $Y$ and the action of $\Sigma_6$ on the irreducible components, we
introduce the following notation. Let $\{\{i,j\},\{k,l\},\{m,n\}\}$ be a partition of $\{0,\ldots,5\}$,
so $\{i,j,k,l,m,n\}=\{0,\ldots,5\}$. Notice that there are $15$ such partitions.
Then, with three choices for the sign, we define planes in $\PP^5$ by:
$$
V_{i\pm j,k\pm l,m\pm n}:\quad Z_i\pm Z_j\,=\,Z_k\pm Z_l\,=\,Z_m\pm\,Z_n\,=\,0\qquad(\subset\PP^5)~.
$$
Notice that besides being symmetric in the variables $Z_i$, the polynomial $F_6$ which defines $Y$ is
also invariant under the change of sign of an even number of variables.

\begin{prop}\label{singY}
The singular locus of $Y$ is the union of $60$ planes. 
There are two $\Sigma_6$-orbits, of length 15 and 45 respectively of these planes, they are the orbits 
of
$$
V_{0-1,2-3,4-5}\quad\mbox{and}\quad V_{0+1,2+3,4-5}
$$
respectively. Let $\beta\Sigma_5\subset\Sigma_6$ be the subgroup, isomorphic to $\Sigma_5$,
generated by the permutations $\beta_{12}$, $\beta_{23}$, $\beta_{34}$, $\beta_{45}$ 
from Lemma \ref{auts5}. Then $\beta\Sigma_5$ has four orbits on the set of sixty singular planes of $Y$, 
they are the orbits of
$$
V_{0-1,2-3,4-5},\quad V_{0-1,2-4,3-5},\quad V_{0+1,2+3,4-5},\quad V_{0+1,2+4,3-5}
$$
and these orbits have length $5$, $10$, $15$ and $30$ respectively.
\end{prop} 

\begin{proof}
A Magma~\cite{Magma} computation gives the irreducible components of the singular 
locus of $Y$ and
the rest are straightforward verifications (done with Magma as well).
\end{proof}

A plane $V\subset\PP^5$ is the projectivization of a linear subspace 
$\tilde{V}\subset\CC^6$ of dimension $3$ and it is
determined by the line $\bigwedge^3\tilde{V}\subset\bigwedge^3\CC^6$ 
(equivalently, by the point in the Grassmannian $Gr(3,6)\subset\PP(\bigwedge^3\CC^6)$).
The $20$-dimensional vector space $\bigwedge^3\CC^6$
has a natural symplectic form given by 
$(\omega,\theta):=\omega\wedge\theta$ $(\in\bigwedge^6\CC^6=\CC)$,
where we fix a basis of $\bigwedge^6\CC^6$. 
Two planes $V,W\subset\PP^5$ are incident if and only if 
$(\bigwedge^3\tilde{V})\wedge(\bigwedge^3\tilde{W})=0$. 
A set of planes is called a set of incident
planes, if any plane of this set has a non-empty intersection with 
each of the other planes in the set.
In particular, a set of incident planes determines an isotropic subspace in 
$\bigwedge^3\CC^6$. 
The following result, together with Corollary \ref{cor20inc}
solves the problem of O'Grady from \cite{ogrady-incident}.

\begin{prop}\label{20inc} (Proof of Theorem \ref{20 incident})
The union of the two
$\beta\Sigma_5$-orbits of $V_{0-1,2-3,4-5}$ and $V_{0+1,2+3,4-5}$ 
consists of the twenty planes
$V_{0\pm 1,2\pm 3,4\pm 5}$, $V_{0\pm 2,1\pm 4, 3\pm 5}$,
$V_{0\pm 3,1\pm 5,2\pm 4}$, $V_{0\pm 4,1\pm 3,2\pm 5}$ and $V_{0\pm 5,1\pm 2,3\pm 4}$, 
all with an odd number of $-$ signs.
This is a set of 20 incident planes. The span in $\bigwedge^3\CC^6$ of 
$\bigwedge^3V_{0\pm 1,2\pm 3,4\pm 5}$, 
$\bigwedge^3V_{0\pm 2,1\pm 4, 3\pm 5}$,
$\bigwedge^3V_{0\pm 3,1\pm 5,2\pm 4}$,
$\bigwedge^3V_{0\pm 4,1\pm 3,2\pm 5}$ and 
$\bigwedge^3  V_{0\pm 5,1\pm 2,3\pm 4}$
(all with an odd number of $-$ signs) is a Lagrangian subspace $A$ of $\bigwedge^3\CC^6$.
\end{prop}

\begin{proof}
This is again a (Magma) computation. The set of $20$ planes is easy to find.
For example,
$V:=V_{0+1,2+4,3-5}$ is in this set. It has a basis $e_0-e_1$, $e_2-e_4$, $e_3+e_5$, 
where $e_0,\ldots,e_5$ 
denotes the standard basis of $\CC^6$. Thus the line 
$\bigwedge^3\tilde{V}_{0+1,2+4,3-5}$ is spanned by the vector 
$$
- e_0\wedge  e_2\wedge  e_3 +  e_1\wedge  e_2\wedge  e_3 -  e_0\wedge  e_3\wedge  e_4 + 
 e_1\wedge  e_3\wedge  e_4 -  e_0\wedge  e_2\wedge  e_5 +  e_1\wedge  e_2\wedge  e_5 + 
 e_0\wedge  e_4\wedge  e_5 - 
     e_1\wedge  e_4\wedge  e_5
~.
$$    
Now one verifies that these planes are indeed incident and span a $10$-dimensional, 
hence Lagrangian, subspace $A$.
\end{proof}

A general Lagrangian subspace $A\subset \bigwedge^3\CC^6$ determines 
a sextic hypersurface
$Y_A\subset \PP^5$, the EPW sextic defined by $A$, as follows. 
Let $v\in\PP^5$ and let $\tilde{v}\in\CC^6$ be a representative, then 
$$
F_v:=\{\omega\in \bigwedge^3\CC^6:\;\tilde{v}\wedge\omega\,=\,0\,\}
$$
is a Lagrangian subspace of $\bigwedge^3\CC^6$. Define
$$
Y_A[k]\,:=\,\{v\,\in\,\PP^5:\;\dim(\,F_v\cap A\,)\,\geq \,k\,\},
$$
then the EPW sextic $Y_A$ is defined as $Y_A=Y_A[1]$.
The sextic $Y_A$ is singular along the surface $Y_A[2]$ of degree forty and 
along the planes $W\subset\PP^5$ 
such that $\bigwedge^3\tilde{W}\in A$ (\cite[Prop.~3.3]{ogrady-incident}).

\begin{prop}\label{Y is EPW}
Let $A\subset\bigwedge^3\CC^6$ be the Lagrangian subspace from Proposition \ref{20inc}. 
Then the EPW sextic $Y_A$ is the sextic $Y$ from Proposition \ref{equationY}.
Its singular locus consists of the surface $Y_A[2]$, 
which is the union of the forty planes in 
the two $\beta\Sigma_5$ orbits of length $10$ and $30$, 
and the $20$ incident planes from   
Proposition \ref{20inc}.
\end{prop}

\begin{proof}
Let $\omega_1,\ldots,\omega_{10}$ be a basis of $A$ and let
$\theta_1,\ldots,\theta_{15}$ be a basis of $\bigwedge^4\CC^6$.
For each basis vector $e_i$ of $\CC^6$ one computes the matrix $M_i$ of 
the linear map 
$A\rightarrow\bigwedge^4\CC^6$ defined by
$\omega\mapsto e_i\wedge \omega$. 
Then $v=(v_0:\ldots:v_5)\in Y_A$ exactly when the matrix
$M_v:=v_0M_0+\ldots+v_5M_5$ has rank at most $9$, hence all $10\times 10$ minors of 
$M_v$ must be zero.
Since $Y_A$ is either a sextic or is identically zero, it suffices to 
factorize a minor to
find a sextic polynomial defining $Y_A$. 

A convenient basis for the Lagrangian subspace $A$ 
consists of the elements (up to sign) in the $\beta\Sigma_5$-orbit of 
$e_0\wedge e_1\wedge e_2+e_3\wedge e_4\wedge e_5$.
We used the standard basis 
$e_i\wedge e_j\wedge e_k\wedge e_l$ for $\bigwedge^4\CC^6$. The submatrix obtained
from the ten basis elements $e_0\wedge e_i\wedge e_j\wedge e_k$, with
$i,j,k\in\{1,\ldots,5\}$ and $i<j<k$, 
has non-zero determinant and one finds the polynomial
$F_6$ as an irreducible factor.
We already found the singular locus of $Y$ and all planes corresponding to points
in $\PP A$, hence $Y_A[2]$ is the union of the planes as in the proposition. 
\end{proof}

\begin{cor}\label{cor20inc}
The set of $20$ incident planes in Proposition
\ref{20inc} is a complete set of incident planes.
\end{cor}
\begin{proof} 
Any plane incident to all the $20$ incident planes in Proposition \ref{20inc}
corresponds to a point of intersection of $\PP(A)\cap Gr(3,6)$. Since 
$Y$ is an EPW sextic, such a plane is in the singular locus of $Y$ 
(\cite[Prop.~3.3]{ogrady-incident}). But none of the remaining $40$ planes 
in the singular locus of $Y$ corresponds to a point of $A$, so the set of $20$ planes 
is complete.
\end{proof}

\begin{rem}\label{YA4} (Compare with \ref{ime4g}) 
  For the Lagrangian space $A$ from Proposition \ref{20inc} we find
  that $ Y_A[3]=Y_A[4]$.  Moreover, $Y_A[4]$ consists of $16$ points
  and for each $v\in Y_A[4]$ the corresponding linear space
  $\PP(F_v\cap A)$ cuts $G(3,6)\subset \PP(\bigwedge^3 \C^6)$ in five
  points.  One of these points is $(1:\ldots:1)\in Y_A[4]$ (see the end of
  the proof of Proposition~\ref{equationY}). 
  This is the point where five planes of type $V_{i-j,k-l,n-m}$
    (three $-$ signs) meet. 
    \end{rem}
    
\begin{rem}\label{rem_planes_intersection}
    Given one of five
    partitions $\{\{i,j\},\{k,l\},\{n,m\}\}$, used to describe 20 incident planes in~\ref{20inc}, and
    choosing one of the three pairs in this partition, one produces a
    point of incidence of another 5-tuple of planes. For example, for
    a partition $\{\{0,2\},\{1,4\},\{3,5\}\}$ we choose the second
    pair $\{1,4\}$ and take the point $(1:-1:1:1:-1:1)$ (minus sign
    with first and fourth coordinate). Then this point is contained in
    5 planes $V_{0+1,2-3,4+5}$, $V_{0-3,1+5,2+4}$, $V_{0+4,1+3,2-5}$,
    $V_{0-5,1+2,3+4}$ and $V_{0-2,1-4,3-5}$. This way we get
    $16=1+5\cdot 3$ points in $Y_A[4]$, each of the points contained
    in a 5-tuple of planes coming from the five different partitions
    from \ref{20inc}. 

    Similarly one finds 30 points where planes associated to the same
    partition meet pairwise. For example, for a partition
    $\{\{0,2\},\{1,4\},\{3,5\}\}$ we choose the second pair. Then
    $V_{0-2,1-4,3-5}$ meets $V_{0+2,1-4,3+5}$ at $(0:1:0:0:1:0)$ while
    $V_{0-2,1+4,3+5}$ meets $V_{0+2,1+4,3-5}$ at
    $(0:1:0:0:-1:0)$. This shows that the planes from \ref{20inc} meet
    pairwise indeed.

Thus the set of 20 planes is divided into~5 subsets of~4 planes in a natural way: a subset corresponds to a partition $\{\{i,j\},\{k,l\},\{n,m\}\}$ and there are 4 choices of signs such that the number of $-$ signs is odd. Each subset of~4 planes contains 16 points in~$Y_A[4]$, on each plane there are 4 of these points. Any two planes corresponding to different partitions meet in one of the 16 points.

\end{rem}


\begin{rem}\label{YA4-bis}
Note that we can reconstruct the $K3$ surface $S'\subset \PP^6$ starting from $A$ 
(cf.~\cite[\S 4.2]{O1}).
Indeed, let $v\in Y_A[4]$, then the dual space $\PP(F_v\cap A)^{\ast}\subset \PP(F_v)^{\ast}$ 
is a five dimensional linear space contained in $\PP^9$. 
The projective space $\PP(F_v)^{\ast}$  naturally contains a Grassmannian 
$G(2,5)$ that cuts $\PP(F_v\cap A)^{\ast}$ along a smooth Del Pezzo surface $S_5$
of degree $5$. 
From \cite[(4.1.5)]{O1} we know that there is a non-degenerate quadratic form on 
$\PP(F_v\cap A)^{\ast} \subset \PP(\bigwedge^3 \C^6)^{\ast}$ (induced by $A\subset \bigwedge^3 \C^6$). 
 We find that the corresponding quadric cuts $S_5\subset \PP^5$ along ten lines.
The $K3$ surface $S'$ is the double cover of $S_5$ branched along these lines.
Note that $S'\subset \PP^6$ is a non-generic $K3$ surface of degree $10$.
\end{rem}

\section{The resolution of the map $S^{[2]}\dashrightarrow Y$}\label{section map}   
In this section we analyze the rational map $g\colon
S^{[2]}\dashrightarrow Y$ defined in section~\ref{eq1}. In
subsection \ref{flop} we present a sequence of flops that resolves the
indeterminacy of this map; in Proposition \ref{propMain} we obtain a
morphism $\overline{g}\colon \overline{S^{[2]}}\to Y$ such that
$S^{[2]}$ and $\overline{S^{[2]}}$ differ by Mukai flops.  In
subsection \ref{section_structure_morphism} we describe the
ramification locus of $\overline{g}$; we need it to obtain the explicit
desingularisation in Theorem \ref{desingularisation}.  
A consequence of our construction is the description (see Remark
\ref{remEx}) of a symplectic resolution of the singular point $\C^4/G$
discussed in \cite{BS} (i.e.~the fiber of $\overline{g}$ over points
from $Y_A[4]$).

\subsection{Flops}\label{flop}
In order to resolve the map $g$ we need to perform birational transformations such that the divisor
$g^{\ast}(\oo_{\PP^5}(1))$ becomes nef. Let us first describe this divisor.

Let
$\mu \colon H^2(S,\Z) \to H^2(S^{[2]},\Z)$ be the natural morphism of cohomology groups. For (the class of)
a curve $C\subset S$, $\mu(C)$ is the class of the divisor with support $\{ \{p,q\}\colon p\in C,q\in S\} \in S^{[2]}$. We
denote by $\Delta^{[2]}_S$ the diagonal divisor on $S^{[2]}$. The isomorphism of $H^2(S^{[2]},\Z)$ with the lattice $S$
given in (\ref{isoh2}), will be fixed so that $\mu (H^2(S, \Z)) = \Lambda_{K3}$ and $\Delta^{[2]}_S = 2\eta. $
\begin{prop}\label{gu} We have $g^{\ast}(\oo_{\PP^5}(1))=\oo_{S^{[2]}}(\mu(C)-\Delta^{[2]}_S)$.
\end{prop}
\begin{proof}
See \cite[\S 4]{O1}. 
 \end{proof}
Let us describe a sequence of Mukai flops that resolves the map $g$.
Recall that in Section \ref{eq1} we defined the map $g$ 
such that 
on an open part of $S^{[2]}$ it can be described as the composition 
\[
g\colon S^{[2]}\stackrel{\bbeta}{\dashrightarrow} (S')^{[2]} 
\stackrel{\aalpha}{\dashrightarrow} Y\subset \PP^5 .
\]
of rational maps. By $(S')^{[2]}$ we understand the
Hilbert scheme of two points in the smooth part of $S'$.
We shall see that there are two sources of indeterminacy for the map $g$: 
the first is the presence of lines on $S'$ and the second are the nodes of $S'$. 
In order to understand more precisely the map $g$ we need to understand the 
geometry of $S'\subset \PP^6(y_0,\dots,y_6)$.

We take a Del Pezzo surface $S_5$ of degree $5$ contained in the
hyperplane $W=\PP^5\subset\PP^6$ defined by $y_6=0$. Let $\cone$ be
the cone over $S_5$ in $\PP^6$ with vertex $P=(0:\ldots:0:1)$ and let
$Q$ be the quadric with equation $y_6^2=q'(y_0,\ldots,y_5)$ as in
Section~\ref{section k3s}.  Then $Q$ intersects $S_5$ along the union
of the ten exceptional lines on $S_5$, $S'=\cone\cap Q$ and $S'$ is
singular exactly at the $15$ points of intersection of these lines.
The projection $\PP^6\to\PP^5$ with center $P$ induces a $2:1$
\emph{morphism} $\rho\colon S'\to S_5$ ramified along the sum of the
ten lines on $S_5\subset W$.  Since the projection by $\rho$ of a line
in $S'$ is a line in $S_5$, we infer that there are exactly ten lines
on $S'\subset \PP^6$. Denote the set of strict transforms of those
lines on $S$ by $\mathcal{T}=\{l_1,\dots,l_{10}\}$ and the fifteen
exceptional curves on $S$ by $e_1,\dots,e_{15}$.

From the definition of the map $g$ in the proof of Proposition \ref{equationY}
it follows that 
$g$ is well defined on $S^{[2]}$ except possibly in $z\in S^{[2]}$ with $z=\{p,q\}$ such that

\begin{enumerate}
\item both $p,q\in l_i$, i.e.\ $z\in l_i^{[2]}$; we have $10$ such planes,
\item both $p,q\in e_i$, i.e.\ $z\in e_i^{[2]}$; we have $15$ such planes,
\item $p\in e_i$ and $q\in l_j$ where $e_i$ and $l_j$ intersect,
we find a surface denoted by $E_{ij}$ that parameterizes 
the closure of this set of reduced sub-schemes $z=\{p,q\}$. 
We obtain $30$ surfaces  (because there are three exceptional curves cutting a 
given line) in $S^{[2]}$,
each isomorphic to $\PP^1\times \PP^1$ blown-up in one point,
\item $p\in e_i$ and $q\in e_j$ with $e_i,e_j$ mapping to two distinct points in one of the lines 
from $\mathcal{T}$. We obtain $30$ surfaces $F_{ij}=\{\{p,q\}\in S^{[2]}\colon p\in e_i, q\in e_j \}$,
each is isomorphic to $\PP^1\times \PP^1$.
\end{enumerate}

Let
$\mathcal{K}\subset S^{[2]}$ be the union of the $85$ surfaces of type 
(1), (2), (3) and (4) 
described above. Note that the indeterminacy locus of $g$ is contained in $\mathcal{K}$.
We now perform a sequence of flops to obtain a fourfold on which 
the transform of the map $g$ is a morphism. Let us analyze the flops locally on $S^{[2]}$ around the surface $l^{[2]}_1$.

A node of the trivalent Petersen
graph corresponding to the $(-2)$-curve $l_1$ 
meet three other $(-2)$-curves $e_1$, $e_2$, $e_3$ (see the diagram in section \ref{section k3s}). 
This gives an initial configuration of $10$ surfaces in $S^{[2]}$:
$$
l_1^{[2]},\quad e_1^{[2]},\quad e_2^{[2]},\quad e_3^{[2]},
\quad F_{12},\quad F_{23},\quad F_{13},
\quad E_{01},\quad E_{02},\quad E_{03},$$
which can be described as follows:  

$$
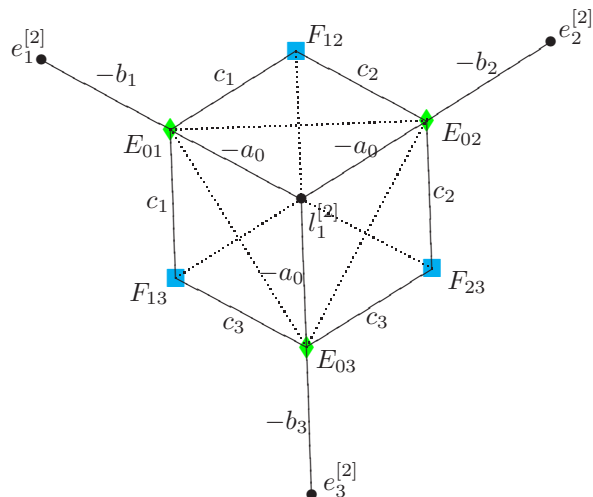
\begin{xy}<56pt,2pt>:
(-1.72,1)*={\psp}="F11", (-1.8,1.1)*={e^{[2]}_1}, 
(0,1)*={\qua}="F12", (0.2,1.1)*={F_{12}}, 
(1.72,1)*={\psp}="F22", (1.9,1.1)*={e^{[2]}_2},
(-0.866,0.5)*={\pspdue}="F01", (0.866,0.5)*={\pspdue}="F02", 
(-1.05,0.4)*={E_{01}}, (1.1,0.4)*={E_{02}}, 
(0,0)*={\psp}="F00", (0.15,-0.15)*={l^{[2]}_1},
(-0.866,-0.5)*={\qua}="F13", (0.866,-0.5)*={\qua}="F23",
(-1.05,-0.6)*={F_{13}}, (1.1,-0.6)*={F_{23}},
(0,-1)*={\pspdue}="F03", (0.2,-1.1)*={E_{03}},
(0,-2)*={\psp}="F33", (0.2,-1.9)*={e^{[2]}_3},
"F11";"F01" **@{-}?(0.5)*!/_3mm/{-b_1}, "F22";"F02" **@{-}?(0.5)*!/^3mm/{-b_2}, "F33";"F03" **@{-}?(0.5)*!/_3mm/{-b_3},
"F12";"F01" **@{-}?(0.5)*!/^2mm/{c_1}, "F12";"F02" **@{-}?(0.5)*!/_2mm/{c_2},
"F13";"F01" **@{-}?(0.5)*!/_2mm/{c_1}, "F13";"F03" **@{-}?(0.5)*!/^2mm/{c_3},
"F23";"F02" **@{-}?(0.5)*!/^2mm/{c_2}, "F23";"F03" **@{-}?(0.5)*!/_2mm/{c_3},
"F00";"F01" **@{-}?(0.5)*!/^2mm/{-a_0}, "F00";"F02" **@{-}?(0.5)*!/_2mm/{-a_0}, "F00";"F03" **@{-}?(0.5)*!/^3mm/{-a_0},
"F00";"F12" **@{.}, "F00";"F13" **@{.}, "F00";"F23" **@{.},
"F01";"F02" **@{.}, "F01";"F03" **@{.}, "F02";"F03" **@{.},
\end{xy}
$$
\par

In this diagram we use the following notation to describe the type of the surfaces: 
$\psp=\PP^2$, $\qua=\PP^1\times\PP^1$ and $\pspdue=\PP^2_2$ is $\PP^2$ 
blown-up in two different points. In the subsequent diagrams we will also
use $\pspuno=\PP^2_1$, $\psptre=\PP^2_3$ to denote blow-up of $\PP^2$ at 
one and three (non-collinear) points, respectively.

Solid line edges of the diagram are intersections of surfaces along curves; 
dotted line edges denote intersections in points. The solid lines edges 
will be labeled by the classes of curves in $Hilb^2(S)$.

Given a curve $C$ on $S$ we have 
a divisor $\mu(C)$ in $Hilb^2(S)$ consisting of
schemes whose support has non-empty intersection with $C$. 
The `diagonal' divisor
$\Delta^{[2]}_S$ in $S^{[2]}$ is the exceptional divisor of the resolution of 
singularities $S^{[2]}\to (S\times S)/\ZZ_2$, where the $\ZZ_2$ action interchanges the factors. 
Outside the
divisor $\Delta^{[2]}_S$ the divisor $\mu(C)$ is isomorphic to 
the complement of $C\times C$ in $C\times S$. 
By $[C]$ we will denote the class of the curve $C\times \{s\}$ where $s\not\in C$.

In what follows $c_0=[l_1]$ and $c_i=[e_i]$, for $i=1, 2, 3$, and $d$ is the class
of a fiber in the blow-up of the diagonal, that is in $\Delta^{[2]}_S$. 
We have the following intersection rules: 
\begin{itemize}
\item $\mu(C)\cdot c_0=1$, and for $i>0$ we have $\mu(C)\cdot
  c_i=\mu(C)\cdot d =0$,
\item $\Delta^{[2]}_S\cdot c_i=0$ for $i\geq 0$ and $\Delta^{[2]}_S\cdot d=-2$.
\end{itemize}
To spare notation in diagrams  
we set 
$$
h=\sum_{i\geq 0} c_i-d,\qquad a_0=d-c_0~,$$ 
and for $j> 0$ we set 
$$
b_j=d-c_j,\quad f_j=d-c_0-c_j,\quad
g_j=h-c_j=\sum_{i\geq 0} c_i-d-c_j,\quad v_j=\sum_{i\geq 0}c_i-c_0-c_j~.
$$
Then for $j> 0$ we have $-\mu(C)\cdot f_j=\mu(C)\cdot
g_j=\mu(C)\cdot h=1$ and $\Delta^{[2]}_S\cdot f_j=\Delta^{[2]}_S\cdot b_j=-2$, $\Delta^{[2]}_S\cdot
g_j=\Delta^{[2]}_S\cdot h=1$, $\mu(C)\cdot b_j=\mu(C)\cdot v_j=\Delta^{[2]}_S\cdot v_j=0$.

\medskip

\par 
We start the process of flopping. 
Note that the lines contained in $e_i^{[2]}$ and $l_j^{[2]}$ have negative intersection with $\mu(C)-\Delta_S^{[2]}$
Note also that among the surfaces from $\mathcal{K}$ only $e_i^{[2]}$ and 
$l_j^{[2]}$ are planes so we have to start the flopping procedure with them. 
Flopping the $e_i^{[2]}$, $i=1,2,3$, each isomorphic to $\PP^2$, 
outside the hexagon we get the picture below on the left and then flopping 
$l_1^{[2]}$, the $\PP^2$ in the center, we get the picture on the right.
We suppress labeling the surfaces and label only the classes of curves 
in the surfaces, notably those which are in the intersections of them.
$$\begin{array}{ccc}
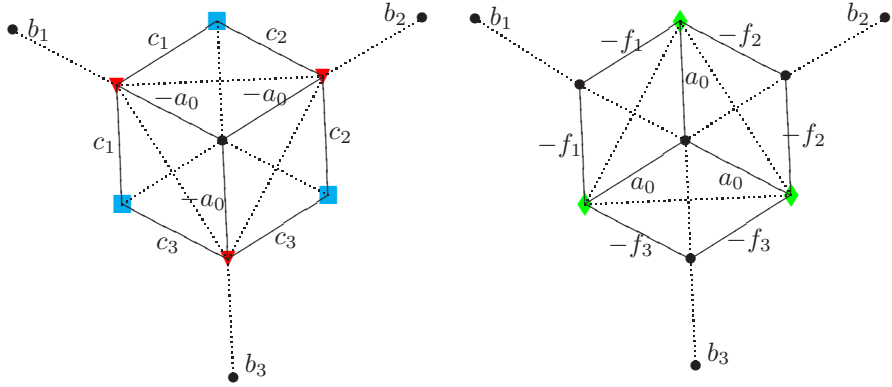
\begin{xy}<45pt,2pt>:
(-1.72,1)*={\psp}="F11", (-1.5,1)*={b_1}, (0,1)*={\qua}="F12", (1.72,0.95)*={\psp}="F22", (1.5,1)*={b_2}, 
(-0.866,0.5)*={\pspuno}="F01", (0.866,0.5)*={\pspuno}="F02", 
(0,0)*={\psp}="F00",
(-0.866,-0.5)*={\qua}="F13", (0.866,-0.5)*={\qua}="F23",
(0,-1)*={\pspuno}="F03",
(0,-2)*={\psp}="F33", (0.2,-1.9)*={b_3},
"F11";"F01" **@{.}; "F22";"F02" **@{.}, "F33";"F03" **@{.},
"F12";"F01" **@{-}?(0.5)*!/^2mm/{c_1}, "F12";"F02" **@{-}?(0.5)*!/_2mm/{c_2},
"F13";"F01" **@{-}?(0.5)*!/_2mm/{c_1}, "F13";"F03" **@{-}?(0.5)*!/^2mm/{c_3},
"F23";"F02" **@{-}?(0.5)*!/^2mm/{c_2}, "F23";"F03" **@{-}?(0.5)*!/_2mm/{c_3},
"F00";"F01" **@{-}?(0.5)*!/^2mm/{-a_0}, "F00";"F02" **@{-}?(0.5)*!/_2mm/{-a_0}, 
"F00";"F03" **@{-}?(0.5)*!/^3mm/{-a_0},
"F00";"F12" **@{.}, "F00";"F13" **@{.}, "F00";"F23" **@{.},
"F01";"F02" **@{.}, "F01";"F03" **@{.}, "F02";"F03" **@{.},
\end{xy}
&&
\begin{xy}<45pt,2pt>:
(-1.72,1.1)*={\psp}="F11", (-1.5,1.1)*={b_1}, (0,1)*={\pspdue}="F12", (1.72,0.95)*={\psp}="F22", (1.5,1)*={b_2},
(-0.866,0.5)*={\psp}="F01", (0.866,0.5)*={\psp}="F02", 
(0,0)*={\psp}="F00",
(-0.866,-0.5)*={\pspdue}="F13", (0.866,-0.5)*={\pspdue}="F23",
(0,-1)*={\psp}="F03",
(0,-1.9)*={\psp}="F33", (0.2,-1.8)*={b_3},
"F11";"F01" **@{.}; "F22";"F02" **@{.}, "F33";"F03" **@{.},
"F12";"F01" **@{-}?(0.5)*!/^2mm/{-f_1}, "F12";"F02" **@{-}?(0.5)*!/_2mm/{-f_2},
"F13";"F01" **@{-}?(0.5)*!/_3mm/{-f_1}, "F13";"F03" **@{-}?(0.5)*!/^2mm/{-f_3},
"F23";"F02" **@{-}?(0.5)*!/^2mm/{-f_2}, "F23";"F03" **@{-}?(0.5)*!/_2mm/{-f_3},
"F00";"F01" **@{.}, "F00";"F02" **@{.}, "F00";"F03" **@{.},
"F00";"F12" **@{-}?(0.5)*!/^2mm/{a_0}, "F00";"F13" **@{-}?(0.5)*!/_2mm/{a_0}, "F00";"F23" **@{-}?(0.5)*!/^2mm/{a_0},
"F12";"F13" **@{.}, "F12";"F23" **@{.}, "F23";"F13" **@{.},
\end{xy}
\end{array}
$$

The left hand side picture below presents the result of flopping the three $\PP^2 $'s 
at the perimeter of the hexagon. Note that the $\PP^2$'s outside the hexagon 
are blown up twice because they are on the link of two such hexagons.
The resulting $\PP^2_2$'s have a common point with the central surface. 
They have three $(-1)$-curves whose classes are in $f_i$, $f_j'$ and $b_j-f_j-f_j'$ 
where
$f_j'$ is the class of the curve coming from the configuration of 
the adjacent node of the Petersen graph. 
The point of intersection of this surface 
with the central $\PP^2_3$ lies on the curve whose class is $f_j$.
The right hand side diagram is obtained by the subsequent flopping of the other 
three $\PP^2$'s at the perimeter. 
In this step the surfaces outside the hexagon are not affected.
\par\medskip
$$\begin{array}{ccc}
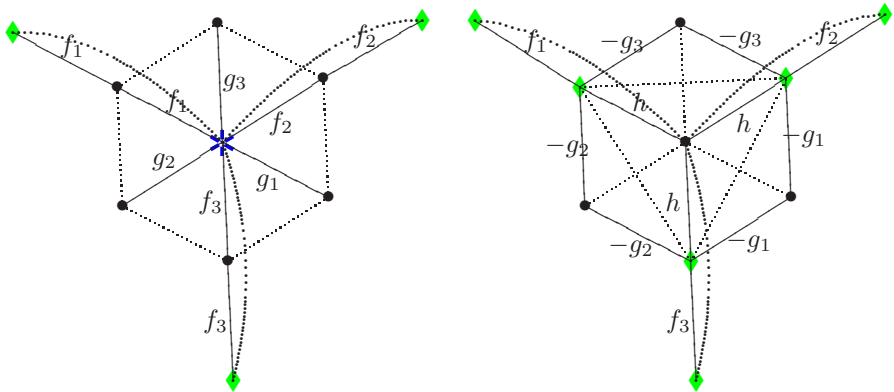
\begin{xy}<45pt,2pt>:
(-1.72,1)*={\pspdue}="F11", (0,1)*={\psp}="F12", (1.72,0.95)*={\pspdue}="F22",
(-0.866,0.5)*={\psp}="F01", (0.866,0.5)*={\psp}="F02", 
(0,0)*={\psptre}="F00",
(-0.866,-0.5)*={\psp}="F13", (0.866,-0.5)*={\psp}="F23",
(0,-1)*={\psp}="F03",
(0,-2)*={\pspdue}="F33",
"F11";"F01" **@{-}?(0.5)*!/_2mm/{f_1}; "F22";"F02" **@{-}?(0.5)*!/^2mm/{f_2}, "F33";"F03" **@{-}?(0.5)*!/_2mm/{f_3},
"F12";"F01" **@{.}, "F12";"F02" **@{.},
"F13";"F01" **@{.}, "F13";"F03" **@{.},
"F23";"F02" **@{.}, "F23";"F03" **@{.},
"F00";"F01" **@{-}?(0.5)*!/^2mm/{f_1}, "F00";"F02" **@{-}?(0.5)*!/^2mm/{f_2}, "F00";"F03" **@{-}?(0.5)*!/^2mm/{f_3},
"F00";"F12" **@{-}?(0.5)*!/^2mm/{g_3}, "F00";"F13" **@{-}?(0.5)*!/^2mm/{g_2}, "F00";"F23" **@{-}?(0.5)*!/^2mm/{g_1},
"F11";"F00" **\crv{~*=<2pt>{.} (-0.9,1)},"F22";"F00" **\crv{~*=<2pt>{.} (1,1)},
"F33";"F00" **\crv{~*=<2pt>{.} (0.3,-1)},
\end{xy}
&&
\begin{xy}<45pt,2pt>:
(-1.72,1.1)*={\pspdue}="F11", (0,1)*={\psp}="F12", (1.72,1)*={\pspdue}="F22",
(-0.866,0.5)*={\pspdue}="F01", (0.866,0.5)*={\pspdue}="F02", 
(0,0)*={\psp}="F00",
(-0.866,-0.5)*={\psp}="F13", (0.866,-0.5)*={\psp}="F23",
(0,-1)*={\pspdue}="F03",
(0,-2)*={\pspdue}="F33",
"F11";"F01" **@{-}?(0.5)*!/_2mm/{f_1}; "F22";"F02" **@{-}?(0.5)*!/^2mm/{f_2}, "F33";"F03" **@{-}?(0.5)*!/_2mm/{f_3},
"F12";"F01" **@{-}?(0.5)*!/^2mm/{-g_3}, "F12";"F02" **@{-}?(0.5)*!/_2mm/{-g_3},
"F13";"F01" **@{-}?(0.5)*!/_2mm/{-g_2}, "F13";"F03" **@{-}?(0.5)*!/^2mm/{-g_2},
"F23";"F02" **@{-}?(0.5)*!/^2mm/{-g_1}, "F23";"F03" **@{-}?(0.5)*!/_2mm/{-g_1},
"F00";"F01" **@{-}?(0.5)*!/^2mm/{h}, "F00";"F02" **@{-}?(0.5)*!/^2mm/{h}, "F00";"F03" **@{-}?(0.5)*!/^2mm/{h},
"F00";"F12" **@{.}, "F00";"F13" **@{.}, "F00";"F23" **@{.},
"F11";"F00" **\crv{~*=<2pt>{.} (-0.9,1)},"F22";"F00" **\crv{~*=<2pt>{.} (1,1)},
"F33";"F00" **\crv{~*=<2pt>{.} (0.3,-1)},
"F01";"F02" **@{.}, "F01";"F03" **@{.}, "F02";"F03" **@{.},
\end{xy}
\end{array}
$$

Now we flop at the central surface.
$$
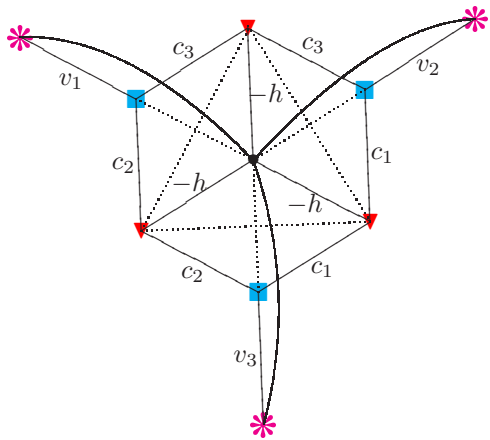
\begin{xy}<50pt,2pt>:
(-1.72,1)*={\pspquat}="F11", (0,1)*={\pspuno}="F12", (1.72,1)*={\pspquat}="F22",
(-0.866,0.5)*={\qua}="F01", (0.866,0.5)*={\qua}="F02", 
(0,0)*={\psp}="F00",
(-0.866,-0.5)*={\pspuno}="F13", (0.866,-0.5)*={\pspuno}="F23",
(0,-1)*={\qua}="F03",
(0,-2)*={\pspquat}="F33",
"F11";"F01" **@{-}?(0.5)*!/^2mm/{v_1}; "F22";"F02" **@{-}?(0.5)*!/_2mm/{v_2}, 
"F33";"F03" **@{-}?(0.5)*!/_2mm/{v_3},
"F12";"F01" **@{-}?(0.5)*!/^2mm/{c_3}, "F12";"F02" **@{-}?(0.5)*!/_2mm/{c_3},
"F13";"F01" **@{-}?(0.5)*!/_2mm/{c_2}, "F13";"F03" **@{-}?(0.5)*!/^2mm/{c_2},
"F23";"F02" **@{-}?(0.5)*!/^2mm/{c_1}, "F23";"F03" **@{-}?(0.5)*!/_2mm/{c_1},
"F00";"F12" **@{-}?(0.5)*!/^2mm/{-h}, "F00";"F13" **@{-}?(0.5)*!/^2mm/{-h}, 
"F00";"F23" **@{-}?(0.5)*!/^2mm/{-h},
"F00";"F01" **@{.}, "F00";"F02" **@{.}, "F00";"F03" **@{.},
"F11";"F00" **\crv{(-0.9,1)},"F22";"F00" **\crv{(1,1)},
"F33";"F00" **\crv{(0.3,-1)},
"F12";"F13" **@{.}, "F12";"F23" **@{.}, "F23";"F13" **@{.},
\end{xy}$$
\par\medskip
Here $\pspquat$ denotes $\PP^2_2$ blown-up at two points at two non-meeting $(-1)$-curves, 
which then become $(-2)$-curves whose classes are in $v_j=f_j+h=\sum_ic_i-c_0-c_j$, 
and $v_j'$, respectively. This surface has also five $(-1)$-curves with classes 
$b_j-f_j-f_j'$, $h$ and $b_j-f_j-h$ as well as  $h'$ and $b_j'-f_j'-h'$.
\par

\begin{lemm}\label{curve}
  The divisor $\mu(C)-\Delta^{[2]}_S$ is nef on this configuration of
  surfaces. It is $\oo(1)$ on the central $\PP^2$, it is trivial on
  $\PP^1\times\PP^1$'s ($\qua$ on the last picture) and defines a ruling on the remaining 6
  surfaces. Thus some multiple of $\mu(C)-\Delta^{[2]}_S$ defines the
  contraction of the configuration of these surfaces to a
  configuration of lines on the image of the plane $\overline{l^{[2]}_0}$.
\end{lemm}

\begin{proof}
  We check that $(\mu(C)-\Delta^{[2]}_S)\cdot
  (-h)= 1$ while the intersection of $\mu(C)-\Delta^{[2]}_S$
  with $c_j$, $v_j$, $f_j+h$, $b_j-f_j-h$, $b_j-f_j-f_j'$ is zero.
\end{proof}

\textbf{Summary.}
We constructed above a sequence of Mukai flops of $\PP^2$'s determined by the 
following classes of 1-cycles: (1) $b_j$'s, (2) $a_i$'s, (3) $f_{ij}$'s, (4) $g_{ij}$'s, 
(5) $h_i$'s. 
Here $i$ is among indices parametrizing vertices and $j$ is 
among indices parameterizing edges of the Petersen graph. 
After flopping those classes $\mu(C)-\Delta^{[2]}_S$ becomes nef (on the configuration of the strict transforms 
of the surfaces in question).
The fourfold $\overline{S^{[2]}}$ is obtained from $S^{[2]}$
by performing this sequence of Mukai flops in the five families of surfaces and
$\overline{g}\colon \overline{S^{[2]}}\to Y$ is the map obtained from
$g:S^{[2]}\dashrightarrow Y$.
We infer the following:
\begin{prop}\label{propMain} The strict
  transform of the complete linear system $|\mu(C)-\Delta^{[2]}_S|$ on
  $\overline{S^{[2]}}$ is big and nef and define a morphism $\overline{g}\colon
  \overline{S^{[2]}} \to Y\subset \PP^5$ that resolves the indeterminacy of the map $g$.
  In particular $g$ is defined by the 
complete linear system $|\mu (C)-\Delta_S^{[2]}|$ and the set $\mathcal{K}$ is the indeterminacy locus
of $g$. 
\end{prop}
\begin{proof}
From Proposition \ref{gu} the map $g$ is defined by a five dimensional linear subsystem of $|\mu(C)-\Delta^{[2]}_S|$. 


  We are flopping $S^{[2]}$ step by step as described in the diagrams
  above.  Also, we saw how the classes of curves changes after our
  flops.  At each step we are doing Mukai flops of a set of disjoint
  $\PP^2$'s corresponding to a curve class that has negative
  intersection with the strict transform of $\mu(C)-\Delta^{[2]}_S$.
    By Lemma \ref{curve} we
  deduce that the proper transform $H$ of
  $\mu(C)-\Delta^{[2]}_S$ on $\overline{S^{[2]}}$ is big and nef.
  
  The Beauville degree $q(H)=q(\mu(C)-\Delta^{[2]}) = q(\mu(C))-q(2\eta) = 10-8 = 2$, so in particular
$\chi(\oo_{\overline{S^{[2]}}}(H)) = 6$ thus $h^0(\oo_{\overline{S^{[2]}}}(H)=6$.
   It follows that the map $\overline{g}\colon \overline{S^{[2]}} \to \PP^5$ defined by $|H|$ is a
  morphism. Since $\overline{S^{[2]}}$ and $S^{[2]}$ are isomorphic in codimension $1$
the morphism $\overline{g}\colon S^{[2]}\to Y\subset \PP^5$ give the resolution of $g$.
It follows also that the map $g\colon S^{[2]}\dasharrow \PP^5 $ is given by the complete linear system 
$|\mu(C)-\Delta^{[2]}_S|$ on 
$S^{[2]}$.
  
  The set $\mathcal{K}$ is in the base locus because it is covered by
  curves with negative intersection with
  $\mu(C)-\Delta^{[2]}_S$. Outside $\mathcal{K}$ there are no base
  points since the map $\aalpha$ so $g$ is well defined there.
\end{proof}

\subsection{The structure of the map $\overline{S^{[2]}}\to Y$} \label{section_structure_morphism}

In this subsection we describe technical results needed in 
the proof of Theorem \ref{desingularisation}. 
Our aim is to give a description of the ramification locus of the map
$\overline{g}$. Our plan is to first describe the ramification of the map $g$. Then to look how this ramification locus transforms by flops, described in Section \ref{flop}, transforming $g$ to $\overline{g}$. 

We consider the Stein factorization 
\[ \overline{g}\colon \overline{S^{[2]}}\xrightarrow{c} Z \xrightarrow{f} Y\subset \PP^5~.\]

Let us first identify twenty divisors $B_1,\dots,B_{20}$ on $\overline{S^{[2]}}$ 
which are contracted by $c$ to singular surfaces on $Z$.
Then we identify the ramification locus of $f$ as the image by $c$ of 
forty surfaces that are the
strict transforms of some surfaces from $S^{[2]}$ isomorphic to $\PP^1\times \PP^1$. 

\begin{lemm}\label{rho} 
  The covering involution of $f\colon Z\to Y$ is induced by the map $\rho\colon S\to S_5$ i.e.~if $\rho(x_1)=\rho(x_2)$ and $\rho(y_1)=\rho(y_2)$, such that $x_1\neq x_2$ and $y_1\neq y_2$ are generic points on $S$, then we have $f(c((x_1,y_1)))=f(c((x_2,y_2))$.
\end{lemm}

\begin{proof} It suffices to study the involution induced by 
  $\aalpha \colon (S')^{[2]}\to Y$.
  The linear system of quadrics containing $S'$ is generated by 
  the five dimensional system of quadrics containing 
  the cone $\cone$ and the quadric $Q$. 
  For a pair $\{p,q\}\in (S')^{[2]}$ consider another pair $(p',q')\in (S')^{[2]}$ such that 
  $p'$ (resp.~$q'$) is the second point of intersection of the line $Pp$ (resp.~$Pq$) with $Q$,
  since the quadrics containing $\cone$ have the property:
  when they vanish on the line $pq$ then they vanish on the line $p'q'$.
  Moreover, $Q$ vanishes at the points $p,q,p',q'$, so the proof is finished.
\end{proof}

 We can now describe the ramification locus of the map $\overline{g}$.
 Let $L_{ij}\simeq\PP^1\times \PP^1$ with $i\neq j$ be the strict transform on $\overline{S^{[2]}}$ 
 of the surface  $L_{ij}'=\{\{p,q\}\in S^{[2]}\colon p\in l_i, q\in l_j \}$ 
 (where $l_i$ for $i=1,\dots,10$ are the strict transforms of the lines from $S'$ on $S$). 
 Note that $L_{ij}'$ is isomorphic to $\PP^1\times \PP^1$ and that there are $45$ such surfaces.
 
 \begin{cor}\label{branch} 
   The branch locus of the map $\overline{g}$ consists of $40$
   surfaces on $\overline{S^{[2]}}$. Each such surface is an element
   of one of the following sets:
\begin{itemize}
\item the set of $30$ surfaces $L_{ij}\simeq\PP^1\times \PP^1$ from
  $S^{[2]}$ for $i\neq j$ such that the lines $\nu(l_i)$ and
  $\nu(l_j)$ do not intersect.
\item The set of ten planes $\overline{l_{i}^{[2]}}\subset
  \overline{S^{[2]}}$ which are the strict transforms of the ten
  planes $l_i^{[2]}\subset S^{[2]}$.
\end{itemize}
\end{cor}
 
\begin{proof}
  The surfaces $L_{ij}'$ are invariant with respect to $\aalpha$, so
  it is enough to show that they are not contracted by $\overline{g}$.
  First observe that $L_{ij}$ is isomorphic to $\PP^2$ blown up at two
  points. Indeed, consider the $\PP^3$ which is the span of two
  disjoint lines $\nu(l_i)$ and $\nu(l_j)$ on $S_5$.  It cuts
  $S_5$ along three lines $\nu(l_i)$, $\nu(l_j)$ and $\nu(l_k)$
  such that $\nu(l_k)$ cuts each $\nu(l_i)$ and $\nu(l_j)$ in
  one point $\nu(A_i)$ and $\nu(A_j)$.  We can deduce that the
  restriction of the map $S^{[2]}\dasharrow \overline{S^{[2]}}$ to
  $L_{ij}$ is the blow-up of the point $(A_i,A_j)\in L_{ij}$
  (corresponding to the intersection with the line $e_ie_j$).
 
  By Lemma \ref{curve} the strict transform of the system
  $|\mu(C)-\Delta_S^{[2]}|$ restricted to the plane
  $\overline{l_i^{[2]}}$ is the system $\mathcal{O}_{\PP^2}(1)$.
  Since $\overline{l_i^{[2]}}$ maps through $c$ to a plane we infer
  that $\overline{l_i}^{[2]}$ is in the ramification locus.
 \end{proof}
 
\begin{rem}
  The images of fifteen of the 45 surfaces $L_{ij}$ that are not in
  the ramification locus of $\overline{g}$ are mapped to points by
  $\overline{g}$.  Indeed, in the case when $\nu(l_i)$ and
  $\nu(l_j)$ intersect, a quadric containing $S'$ either cuts the
  plane spanned by $\nu(l_i)$ and $\nu(l_j)$ along these two lines
  or it contains this plane.  The definition of the map $\aalpha$
  implies that $\overline{g}$ maps $L_{ij}$ to a point.
\end{rem}
  
Let us identify the exceptional divisors on $\overline{S^{[2]}}$ that
are mapped to surfaces on $Z$ by $c$. The exceptional divisors $B_i$
will be the strict transforms on $\overline{S^{[2]}}$ of divisors
$B_i'\subset S^{[2]}$ defined in the following way.
Fifteen of them are easy to describe. Let 
$B_i'=\{\{p,q\}\in S^{[2]} \colon p\in e_i \}$, 
where $e_i$ is one of the $15$ curves contracted by 
the map $S\rightarrow S'$, then
these divisors are already contracted by 
$\bbeta\colon S^{[2]}\to (S')^{[2]}$. 

Let us find the remaining five divisors $B'_{16},\dots,B'_{20}$.    
There are five pencils of conics on $S_5\subset \PP^5$ 
(if $S_5\to \PP^2$ is the blow-up in 4 points 
the pencils correspond to lines passing through these points and the conics through the 4 points). 
These pencils induce five elliptic pencils $\lambda_1,\dots,\lambda_5$ on $S$. 
We define the divisors 
$B_{15+i}'=\{ \{p,q\} \in S^{[2]}: \exists K\in \lambda_i\colon p\in K, q\in K  \}$.
Each of divisors $B_{15+i}'$ has a fibration $B_{15+i}'\rightarrow \PP^1$ with fibers of type $K^{[2]}$.

Denote by $B_1,\dots, B_{20}$ the strict transforms of 
$B'_1,\dots,B'_{20}$ on $\overline{S^{[2]}}$.
We know from Proposition \ref{Y is EPW} that $Y$ is singular along $60$ planes.
We shall see that the twenty of them described in Proposition \ref{20inc} 
are the images of $B_1,\dots,B_{20}$.
Note that in Proposition \ref{singE^4} we prove that the images of 
$B_1,\dots,B_{20}$ on $Z$ are singular $K3$ surfaces with normalization being the 
Vinberg K3 surface.
\begin{prop}\label{singZ} The divisors $B_i$ for $i=1,\dots, 20$ are
  contracted through $c$ to surfaces in $Z$ such that the images of
  the $B_i$ on $Z$ intersect each other.  
    A general fiber of $c$ in the divisors $B_1,\dots,
    B_{15}$ is a curve of type
    $c_i$, with $i>0$, as defined in section~\ref{flop}.
Moreover, there are no other divisors on $S^{[2]}$ 
that are contracted to surfaces by $g$. 
\end{prop}

\begin{proof}
Let $u$ be the strict transform on $S$ of a conic in $S_5\subset \PP^5$. 
It is enough to prove that 
the surface $u^{[2]}\subset S^{[2]}$ maps to a line in $Z$.
First, the curve $u$ is contained in $\PP^3_u$ 
(spanned by $P$ and the plane spanned by the conic in $S_5$).
The curve $u$ is the intersection of two quadrics: $Q$ and the cone with vertex $P$.
When $u$ is chosen generically, the quadric $Q\cap \PP^3_u$ is smooth and thus is isomorphic to 
$\PP^1\times \PP^1$. Each line on $Q$ cuts $S'$ in two points such that the two rulings on $Q$ 
define two curves on $S^{[2]}$. 
By the description of the map $\aalpha$ we see that both these curves are 
contracted by $g$ to the same point on $Y$. 

Finally, we find explicitly points in the intersections of each pair of the divisors $B_i$, $B_j$ for 
$0\leq i<j\leq 15$.

In order to prove that there are no more contracted divisors, 
it is enough to observe that each such divisor maps to
a plane in the singular locus of $Y\subset \PP^5$. On the other hand, we know that~$Y$ is an EPW sextic
and from \cite{ogrady-incident} the contracted divisors $B_i$ for $1,\dots,20$ are mapped to a maximal set of incident planes. 
There are no more contracted divisors since the image of such a divisor would 
be a plane which intersects all the twenty planes above.
\end{proof}

\section{A Kummer type IHS and the Debarre-Varley p.p.a.v.}\label{secPolarization}
The IHS fourfold $X_0$ constructed by 
Donten-Bury and Wi\'{s}niewski in \cite{DW} 
is a quotient of a principally polarized abelian fourfold 
which we study in detail in this section. 
In Theorem~\ref{desingularisation} we show that $X_0$ is birationally 
equivalent to~$S^{[2]}$.

\subsection{Polarization}\label{DW IHS}
The variety $X_0$ is constructed as a desingularisation of a quotient $E^4/G$ where
$E=\C /(\Z+i\Z)$ is the elliptic curve with complex multiplication by 
$\Z[i]$ and $G\cong Q_8\times_{\Z_2}D_8$ is a subgroup of $Aut(E^4)$.
 
Recall that the action of $G$ on $E^4$ is given by the matrices $T_j$ for $j=0,\ldots,4$, listed in Section~\ref{aute4} below (cf.~\cite[\S 6B]{DW}). 
They satisfy the relations $T_j^2=I$,
$T_jT_k=-T_kT_j$, and thus $(T_jT_k)^2=-I$ for $j\ne k$.

The abelian fourfold $E^4$ also has as automorphisms
$$
i\colon E^4\,\longrightarrow E^4,\qquad  (x,y,z,t)\,\longmapsto\, (ix,iy,iz,it)~,
$$
and $(-1):=i^2$.

Let us find a polarization $H\in NS(E^4)$ on the abelian fourfold $E^4=\C^4/ \Lambda$, 
where $\Lambda=\Z[i]^4$, which is invariant with respect to $G$.
By \cite[\textsection 2.2]{BL} this is equivalent to finding a
$G$-invariant Hermitian  $4\times 4$ matrix $H$
with coefficients in $\Z[i]$.

\begin{prop}\label{herm form}
Any $G$-invariant Hermitian matrix with coefficients in $\Z[i]$
has the following shape, where $a\in \Z$:
$$
H_a\,:=\,a\begin{pmatrix} 2&0&1&1+i \\
0&2&1+i&-i\\
1&1-i&2&0\\
1-i&i&0&2
\end{pmatrix}~.$$
\end{prop}

\begin{proof}
The $G$-invariant Hermitian matrices satisfy
equations $T_j\cdot H\cdot \overline{T}_j^t=H$,
for each $0\leq j\leq 4$,
where $\overline{T}_j^t$ is the
transposed of the complex conjugate of $T_j$. Notice that these are linear equations
in the coefficient of $H$.
\end{proof}

We then find that $H:=H_1$ is positive definite and since $\det H=1$, 
it defines a principal polarization on $E^4$.

It turns out that the principally polarized abelian variety $(E^4,H)$
was known before, see \cite{D,V}.
In fact, Debarre in \cite{D} proved that there exists a unique indecomposable p.p.a.v.\
$(A_{10},L)$ of dimension $4$ that is not a hyper-elliptic Jacobian and admits the
maximal number of $10$ vanishing theta-constants
(points of order two on a symmetric theta divisor which are singular with even multiplicity).

\begin{prop}\label{abelian} The abelian fourfold $(E^4,H)$ is isomorphic as a p.p.a.v. to the
Debarre-Varley p.p.a.v. $(A_{10},\Theta)$.
In particular, the singular locus of the corresponding theta divisor consists of $10$ 
ordinary double points and if theta divisor is chosen to be symmetric, 
these ODPs are two-torsion points of $E^4$.
\end{prop}

\begin{proof}
The real part of $H$ defines a $\Z$-valued quadratic form on $\Z^8=\Z[i]^4$
and one finds that it is a positive definite even unimodular quadratic form. Hence,
in a suitable basis, it
must be the quadratic form associated to the root system $E_8$.
Now the proposition follows from the construction in \cite[\S 5]{D}.
\end{proof}

\subsection{Invariant line bundles}\label{inv lb}
We will show that there are exactly $16$ $G$-invariant line bundles on $E^4$ whose
first Chern class is the alternating form $E={Im}\, H$.
First we consider the fixed points of the action of $G$ on $E^4$.

\begin{lemm}\label{lem-fixed}
The subgroup $(E^4)^G$ of points of $E^4$ which are fixed by $G$ is
isomorphic to $(\Z_2)^4$.
The $16$ fixed points are $(a_1,a_2,a_3,a_4)$ where $a_j$ is either $0$ or $\frac{1+i}{2}$.
Moreover,
these points are also the fixed points of the automorphism $i$ of $E^4$.
\end{lemm}

\begin{proof}
As $(T_iT_j)^2=-I$ if $i\neq j$, a fixed point of $G$ is a point
$x\in E^4$ with $-x=x$ that is, $2x=0$. Now one checks that of the $2^8=256$ two-torsion  points in $E^4$ only the $16$ points given in the lemma are fixed by $G$.
Similarly, the fixed points for $i$ must be two-torsion points and one finds the same $16$ points.
\end{proof}

The line bundles on an abelian variety $\C^d/\Lambda$ with a given first Chern class $E$ and
a $\Z$-valued alternating form on $\Lambda$,
are parametrized by semi-characters, i.e.\ by maps $\alpha\colon \Lambda \to \C_1$ with $\C_1=\{z\in \C \colon |z|=1 \}$ the circle group, 
satisfying $$\alpha(x+y)=\alpha(x)\alpha(y)e(x,y)$$ where $e(x,y):=\exp(\pi i E(x,y))$ (see \cite[\textsection 2.2]{BL}).
Notice that a semi-character is completely determined by its values on a $\Z$-basis of
$\Lambda$.

The line bundle  $L_\alpha$ defined by $\alpha$ is $G$-invariant
if  and only if for each $g\in G$ we have $\alpha(g(x))=\alpha(x)$ and it is
symmetric, so $(-1)^*L_\alpha\cong L_\alpha$, if and only if $\alpha(\Lambda)\subset\{\pm 1\}$.
The semi-character of a symmetric line bundle factors over $\Lambda/2\Lambda$ and this group is naturally isomorphic to the group of two-torsion points on  the abelian variety.

\begin{prop}\label{automorphism i}
There are exactly $16$ $G$-invariant line bundles on $E^4$ whose
first Chern class is the alternating form $Im\, H$. These $16$ line bundles are symmetric
and are also invariant under the automorphism $i$ of $E^4$. The corresponding semi-characters
$\alpha:E[2]^4\rightarrow\{\pm 1\}$ are exactly those for which $\alpha(x)=1$ for all
$x\in (E^4)^G$.
\end{prop}

\begin{proof}
Since $-I\in G$, any $G$-invariant line bundle is symmetric.
Hence its semi-character takes values in $\{\pm 1\}$.
To find such semi-characters with $\alpha(T_j(x))=\alpha(x)$ for all $x\in \Lambda$
and each $0\leq j\leq 4$,
we use the  $\Z$-basis of $\Lambda=\Z[i]^4$ given by eight vectors $(1,0,0,0),  (i,0,0,0),\ldots,(0,0,0,i)$.
By computations one finds that the $G$-invariant
semi-characters are those that have the values $(a_1,a_1,a_2,a_2,a_3,a_3,a_4,a_4)$,
with $a_i\in\{\pm 1\}$, on these basis vectors.
In particular, there are $16$ of these and they satisfy $\alpha(ix)=\alpha(x)$.

Then it is easy to check that for
$x_1=(1+i,0,0,0), \ldots, x_4=(0,0,0,1+i)$
one has
$e(x_i,x_j)=+1$, for $1\leq i,j\leq 4$, hence the Weil pairing is trivial on
$(E^4)^G$. One also easily verifies that
these $16$ G-invariant semi-characters satisfy $\alpha(x_i)=1$, for $i=1,\ldots,4$,
hence $\alpha(x)=1$ for all $x\in (E^4)^G$.
Conversely, a semi-character with values in $\{\pm1\}$ which is trivial on $(E^4)^G$
is completely determined by its values on the four vectors
$(1,0,0,0), \ldots, (0,0,0,1)$. Thus there are exactly $16$ such semi-characters and, using the results above, we conclude that these are the semi-characters of the $G$-invariant line bundles with first Chern class $Im\, H$.
\end{proof}

Since $(E^4,H)$ is a principally polarized abelian variety and $H$ is $G$-invariant, the
isomorphism $\lambda_H:E^4\rightarrow Pic^0(E^4)$ induced by $H$ is a bijection between the
fixed points of $G$ in $E^4$ and the $G$-invariant line bundles with trivial first Chern class.
Tensoring a $G$-invariant line bundle $L$ having $c_1(L)=Im\, H$ with a $G$-invariant line bundle
in $Pic^0(E^4)$, one obtains again a $G$-invariant line bundle with first Chern class $Im\, H$.
Conversely, if also $M$ is $G$-invariant and $c_1(M)=Im\, H$, then $L\otimes M^{-1}$ has trivial first Chern class and is $G$-invariant.

A $G$-invariant line bundle $L$ with first Chern class $Im\,H$ defines a principal polarization, 
hence $\dim H^0(E^4,L)=1$. We denote by $D$ the corresponding theta divisor, that is,
the unique effective divisor $D$ in $E^4$ such that $L=\oo(D)$.
We refer to these $16$ theta divisors as the $G$-invariant theta divisors, each of these
has a singular locus consisting of $10$ ODPs by Debarre's results in \cite{D}. 
Moreover, if $D$ is a $G$-invariant theta divisor and $p\in (E^4)^G$, 
then also $D+p=t_p^*D$ is a $G$-invariant theta divisor.

\subsection{The automorphism group}\label{aute4}
To find the configuration of the $G$-invariant points and divisors, 
we will use the action of the
automorphism group $G_{DV}$ of the p.p.a.v.\ $(E^4,H)$.

In \cite[\S 5]{D} one finds that after choosing (any) $J\in W(E_8)$ 
with $J^2=-1$, one obtains an isomorphism between
the root lattice of $E_8$ and the lattice $\Lambda$ defining $A_{DV}$ such that 
$J$ corresponds to the multiplication by~$i$.
The automorphism group $G_{DV}:={Aut}(E^4,H)$ is the subgroup of the Weyl group 
$W(E_8)$ of elements which commute with $J$. 
The group $G_{DV}$ has order $46080=2^6\cdot (6!)$.
It has a normal subgroup
$\widetilde{G}:=G\times_{\Z_2}\Z/4\Z$ of order $2^6$,
which is the group $(G,i)$ generated by $G$ and $i$.
The quotient group $G_{DV}/\widetilde{G}$ is isomorphic to the symmetric group $S_6$.
The isomorphism between the root lattice of $E_8$ and $\Lambda=\Z[i]^4\subset\C^4$ 
defines a four-dimensional complex representation of $G_{DV}$. 
The invariant theory of this group was studied by Maschke \cite{maschke}.
The representation of $G_{DV}$ on the alternating forms on $\C^4$ permutes,
up to scalar factors, a certain basis of six alternating forms.
This provides the surjective homomorphism $G_{DV}\rightarrow S_6$.
$$
0\,\longrightarrow \, \widetilde{G}\,=\,(G,i)\,
\longrightarrow\,G_{DV}\,\longrightarrow\, S_6\,\longrightarrow\,0~.
$$

The unitary group of the hermitian form $H=H_1$ from Proposition \ref{herm form} is
$$
U(H)\,:=\,\{M\,\in\,GL(4,\Z[i]):\,MH\overline{M}^t\,=\,H\,\}~.
$$
By definition, it is the subgroup of $Aut(E^4)$, fixing $0\in E^4$, 
which preserves the polarization defined by $H$. 
In particular, $U(H)\cong G_{DV}$ and $\sharp U(H)=46080$.

The group $G$, which is a subgroup of $U(H)$,
is generated by the five matrices, 
$T_{j}\,:=\,N_5^jT_0N_5^{-j}$ with $j=0,1,\ldots,4$, where
$$
T_0\,:=\,
\begin{pmatrix}1&0&0&0\\0&-1&0&0\\0&-1+i&1&0\\1-i&0&0&-1\end{pmatrix},\quad
N_5\,:=\,
\begin{pmatrix} 1&-1&0&-i\\-i&0&i&i\\-i&-1&1+i&0\\1&-1+i&0&-1-i\end{pmatrix}.
$$
In particular, $N_5$ normalizes $G$ and we verified that
$N_5\,\in\,U(H)$.
The following matrices are also in $U(H)$:
$$
N_{01}\,:=\,
\begin{pmatrix} 1+i&0&-1-i&-i\\0&i+1&-i&0\\0&1&-1-i&0\\1&-1+i&0&-1-i\end{pmatrix},
\quad
N_{45}\,:=\,
\begin{pmatrix} 0&1&0&i\\0&1+i&-i&0\\0&1&0&0\\-1&0&1&1+i\end{pmatrix}.
$$
One has $N_{01}T_0N_{01}^{-1}=T_1$ and conjugation by $N_{01}$ fixes $T_2,T_3$ 
and maps $T_4$ to $-T_4$. 
The matrices $N_5,N_{01}$ generate a subgroup $N_G$ of order $7680=64\cdot 120$ of $U(H)$ 
which contains the subgroup $G$ as well as multiplication by $i$, 
and the quotient $N_G/(G,i)$ is isomorphic to the symmetric group $S_5$. 

The matrices $N_5,N_{01},N_{45}$ generate the group $U(H)$,
$$
U(H)\,=\,\langle\,N_5,\,N_{01},\,N_{45}\,\rangle.
$$
One has $N_{45}T_jN_{45}^{-1}=-iT_jT_4$ for $j=0,1,2,3$ and
$N_{45}T_4N_{45}^{-1}=T_4$. 
The isomorphism of $U(H)/(G,i)$ with the symmetric group $S_6$
can be seen from the action of $U(H)$ on the alternating forms on $\CC^4$. 
Let
$$
E_5\,:=\,
\begin{pmatrix}0&0&1&0\\0&0&0&1\\-1&0&0&0\\0&-1&0&0\end{pmatrix},\quad
E_4\,:=\,N_{45}E_5N_{45}^t\,=\,
\begin{pmatrix}0&1-i&-i&1+i\\-1+i&0&0&i\\i&0&0&1+i\\-1-i&-i&-1-i&0\end{pmatrix}.
$$
Then $gE_5g^t=E_5$ for all $g\in G$ and the $U(H)$-orbit of $E_5$ 
consists of $6=\sharp U(H)/\sharp (G,i)$ 
matrices, up to sign, one of which is $E_4$. 
One has $T_jE_4T_j^t=-E_4$ unless $j=4$ in which case one finds $ T_4E_4T_4^t=E_4$.

The group $U(H)$ permutes the $16$ $(G,i)$-invariant divisors and there are at least two orbits, 
since such a divisor may or may not contain $0\in E^4$.
Each of the ten $G$-invariant divisors which contain $0\in E^4$ has a node in $0$ and thus has a  
tangent cone which is given by a quadratic form on $T_0E^4=\C^4$. 
These ten quadratic forms are fixed, up to a scalar multiple, by $G$ and they are permuted, 
up to a scalar multiple, by $U(H)$. We define two symmetric matrices
$$
q_{012}\,:=\,
\begin{pmatrix}2&0&1&1-i\\0&2i&1+i&-1\\1&1+i&2&0\\1-i&-1&0&-2i\end{pmatrix},\quad
q_{013}\,:=\,
\begin{pmatrix}0&0&1&0\\0&2i&1+i&-1\\1&1+i&2&0\\0&-1&0&0\end{pmatrix}.
$$
One has $T_jq_{012}T_j^t=\epsilon_jq_{012}$ with 
$\epsilon_j=+1$ for $j\in\{0,1,2\}$ and $-1$ for $j=3,4$.
Thus $q_{012}$ is a common eigenvector for the group $G$ of the 
space of $4\times 4$ symmetric matrices on $\C^4$.
The other nine eigenvectors can be obtained as $N_5^jq_{012}(N_5^j)^t$, 
$N_5^jq_{013}(N_5^j)^t$ where
$j=0,1,\ldots,4$, these two $N_5$-orbits form one $U(H)$-orbit, since 
$N_{01}N_5q_{012}(N_{01}N_5)^t=N_5^2q_{013}(N_5^2)^t$.

Thus the corresponding quadrics are the $10$ tangent cones
to the $G$-invariant divisors passing through $0\in E^4$. 

The subgroup $U(H)_{012}$ of $U(H)$ which fixes $q_{012}$ up to 
scalar multiple has index $10$ in $U(H)$.
We checked that it can be generated by three elements:
$$
U(H)_{012}\,=\,\langle N_{01},\,N_{45},\,N_f\rangle,\qquad 
N_f\,:=\,\begin{pmatrix}i&0&0&0\\1-i&i&0&-2\\-i&0&i&-1+i\\1+i&0&0&-i\end{pmatrix}.
$$

\begin{prop}\label{10 ODP}
Let $D:=D_{012}$ be the $G$-invariant divisor with $0\in D$ 
and with tangent cone defined by $q_{012}$. 
Then the $10$ ODPs on $D$ are the points in $(E^4)^G$ which are {\emph {not}} in the following list
of six points in $(E^4)^G$:
$$
p_1:=[1, 0, 0, 0],\;p_3:=[0, 0, 1, 0],\;p_1+p_3,\quad p_2:=[0, 1, 0, 0],\;p_4:=[0, 0, 0, 1],\;p_2+p_4,
$$
where $[a_1,a_2,a_3,a_4]=\mbox{$\frac{1+i}{2}$}(a_1,a_2,a_3,a_4)$ in $E^4=(\C/\Z[i])^4$.
\end{prop}

\begin{proof}
The $10$ ODPs of $D$ are two-torsion points of $E^4$, and $0$ is one of them.
Thus $9$ of them are non-zero and the automorphism $i$ of $E^4$, 
which maps $D$ into itself, permutes these nine ODPs. 
As $i^2=-1$, which is identity on the two-torsion points, 
at least one ODP is an $i$-fixed point and hence,
by Lemma \ref{lem-fixed}, it is also a fixed point of $G$.
We checked that $U(H)_{012}$ has three orbits on
the fixed points of $G$ in $E^4$, they are $0$, the six points listed above and the 
remaining $9$ points of $(E^4)^G$. 
If the orbit of six consists of ODPs, then there remain $10-1-6=3$ ODPs to be identified, but again 
one of these three must be an $i$-fixed point and then using the $U(H)$-action we get $8$ more ODPs, 
which contradicts that $D$ has only 10 ODPs. Thus the 10 nodes all lie in $(E^4)^G$ and 
there are two $U(H)$-orbits, one of length 1 and one of length~$9$.
\end {proof}

The divisors $2(D+p)$ in $E^4$, with $p\in E^4[2]$,
are all linearly equivalent by the Theorem of the Square
and the linear system $|2D|$ has dimension $2^4-1=15$.
We will prove in Proposition \ref{prop-Di}
that the span of the 16 divisors $2(D+p)$ where $p$ runs over
$(E^4)^G$ is $5$-dimensional.
Here we give an estimate for the span that will be used to deduce this fact.

\begin{prop}\label{span >5}
The sixteen divisors $2(D+p)$, where $p$ runs over $(E^4)^G$,
span a subspace of dimension at least 5 in $|2D|$.
\end{prop}

\begin{proof}
We give a list of $6$ points $q_1,\ldots,q_6$ and six points $r_1,\ldots,r_6$, all in $(E^4)^G$,
such  that $r_1\in D+q_i$ for $i\geq 2$ but $r_1\not\in D+q_1$, and similarly
such that $r_i\in D+q_j$ if $i<j$ but $r_i\not\in D+q_i$, 
which proves the proposition.
Here $D=D_{012}$ and the six points in $(E^4)^G$ which are not in
$D\cap (E^4)^G$ are listed in Proposition \ref{10 ODP}.
Notice that $r_i\in D+q_i$ if and only if $r_i-q_i\in D$.
We take, with the notation from Proposition \ref{10 ODP}, $q_i$ to be:
$$
0,\quad p_3,\quad p_1+p_2+p_4,\quad p_1,\quad p_2+p_3,\quad p_1+p_2+p_3+p_4,
$$
whereas the points $r_i$ are:
$$
p_4,\quad p_2+p_3+p_4,\quad p_1+p_2,\quad p_1+p_2+p_4,\quad p_1+p_2+p_3,\quad p_1+p_3+p_4.
$$
This concludes the proof.
\end{proof}

What emerges from these results is a configuration of $16$ points, those of $(E^4)^G$, 
and the 16 $G$-invariant theta divisors, which are the $D+p$ with $p\in (E^4)^G$.
Each divisor contains exactly 10 of the points. A principally polarized abelian surface $A$ 
also defines a $(16,6)$-configuration (cf. \cite[\textsection 10.2]{BL}), 
consisting of the points of $A[2]\cong(\Z_2)^4$
and the six points (with multiplicity one) on each of the 16 symmetric theta divisors
which can be written as $\Theta_A+p$ for $p\in A[2]$ for a(ny) symmetric theta divisor $\Theta_A$ on $A$.
In case $A=E_1\times E_2$ is a product of two elliptic curves with the product polarization, 
one can take $\Theta_A=E_1\times\{0\}+\{0\}\times E_2$.
This divisor contains the six points $(q_1,0)$ and $(0,q_2)$ with multiplicity one, 
where $q_i\in E_i[2]-\{0\}$ (and it contains $0=(0,0)$, but with multiplicity two). 
So if we use the basis of $(E^4)^G$ from Proposition \ref{10 ODP}, 
(with second and third coordinates permuted) we see that
there is an isomorphism $(E[2]^4)^G\cong A[2]$ such that
the first set of three points not on $D$ is mapped to $E_1[2]-\{0\}$ and the second set is mapped to
$E_2[2]-\{0\}$. Using translations we then find that the configuration defined by the 
$G$-invariant theta divisors of $E^4$ in  the group $(E^4)^G$ is also a $(16,6)$-configuration.

\begin{cor}\label{inters}
The intersection of two distinct $G$-invariant divisors contains exactly 
$6$ points from $(E^4)^G$, thus there are exactly two points in $(E^4)^G$ not contained in their
union.
\end{cor}
\begin{proof}
This is a well-known property of the $(16,6)$-configuration and can also be checked by making an incidence
table (as in \cite[p.787, Fig.~21]{GH}) of the $G$-invariant divisors and the points in $(E^4)^G$. 
\end{proof}

\subsection{The fixed points of $G$}\label{section_geom_quot}
We consider the fixed points of the action of the groups $G$ and $(G,i)$ on $E^4$.
As a consequence we will describe the singular locus of $E^4/G$ and $E^4/(G,i)$.
Recall from \cite[\S 6.B]{DW} that the action of $G$ on $E^4$ 
has the following sets of points where the isotropy group is not trivial:
\begin{itemize}
\item $16$ fixed points with isotropy $G$,
\item $240$ points with isotropy $\Z_2\oplus \Z_2$,
\item $40$ surfaces with isotropy $\Z_2$.
\end{itemize}
The $16+240=256=2^8$ isolated points 
with non-trivial stabilizer are exactly the two-torsion points in $E^4$.
Let us now describe the fixed surfaces more precisely.
The five generators $T_i$ of $G$ and also the $T_{i+5}:=-T_i$ for $i=0,\ldots, 4$
are symplectic reflections, i.e.\ they have exactly two eigenvalues different from~1. 
When acting on $E^4$, each of these ten symplectic reflections $T_i$ fixes 
four disjoint surfaces isomorphic to $E\times E$.
Denote them by $K^i_1,\dots, K_4^i$ for $i=0,\ldots, 9$.

Since $T_i$ and $-T_i$ are conjugate in $G$, after reordering $K^i_1,\dots, K_4^i$
we may assume that $K_j^i$ and $K^{i+5}_j$
have the same image in $E^4/G$ for $j=0,\ldots,4$.
Let us fix one such pair of surfaces $K$ and $K'$ fixed by symplectic reflections 
$T, -T \in G$ respectively.
The action of~$G$ restricted to~$K$ is the action of $N(T)/\langle T\rangle \simeq Q_8$,
described in~\cite[\S 6.A]{DW}. The quotient $Z:=E^2/Q_8$ has three $A_1$
singularities and four $D_4$ singularities and was studied in \cite{DW}.

This means that on $K$ there are three orbits of four points each and 
four points that are fixed by the action of $Q_8$ (so also of $(G,i)$).

\begin{prop}\label{singE^4}
  The singular locus of $E^4/G$ consist of $20$ singular surfaces
  $L_1,\dots, L_{20}$.  Each of these surfaces is a K3 surface
  isomorphic to $Z=E^2/Q_8$.  
\end{prop}

\begin{proof} A singular point of $E^4/G$ is the image of a point of $E^4$ 
with non-trivial isotropy.
We observe that the union of the surfaces $K^i_j$ contains all such points.
As we saw before, the forty fixed surfaces $K^i_j$ are mapped through $\eta \colon E^4 \to   E^4/G$ 
to twenty surfaces.
It can be checked in local coordinates that
the image $\eta (K)=\eta(K')=L \in E^4/G$ is normal. 
It is easy to see that the map $K\to L\in E^4/G$ factorizes through $Z$.
Looking at the orbits of $G$ we conclude that the map $Z\to L$ is a bijection.
\end{proof}

\begin{rem} We shall see in Section \ref{secDesing} that the surfaces $L_1,\dots, L_{20}$ are mapped through the quotient map $E^4/G\to E^4/(G,i)$ to planes contained in the singular locus. In fact, after proving that $E^4/(G,i)=Y\subset \PP^5$ we will see that the above planes will be twenty incident planes
considered in Proposition \ref{20inc}. 
\end{rem}

Since $L_1,\ldots,L_{20}$ are images of $K^i_j$ for $i=0,\ldots, 5$, $j = 1,\ldots, 4$, we obtain a division of this set into~5 subsets of~4 planes. It follows from the desription above that the configuration of their intersection points is as described in Remark~\ref{rem_planes_intersection}. In particular, each subset contains 16 points with isotropy group~$G$, and any two planes from different subsets meet in one of the~16 points.
Note also that the remaining three surfaces fixed by $-T$ cut $K$ 
in three orbits of four elements 
for the action of $Q_8$. We observe that a $Q_8$-orbit of four points on~$K$ 
is a part of a $G$-orbit with eight points on $E^4$ and, more precisely,
given a four element orbit on~$K$ there is a four element orbit on $K'$ 
forming together an eight element orbit of~$G$.

The group $(G,i)$ has $30$ symplectic reflections, since also the $20$ elements
$\pm iT_jT_k$ (for $0\leq j< k\leq 4$)
are reflections.
Each of these reflections has four fixed surfaces on $E^4$. We denote by 
$\F$ the set of the $30\cdot 4=120$ fixed surfaces in $E^4$ 
of the $30$ symplectic involutions in $(G,i)$.

\begin{lemm}\label{singular set}
The fourfold $E^4/(G,i)$ is singular along $60$ surfaces, 
we denote by $\s$ the set of these surfaces. 
The surfaces from $\s$ are the images of the $120$ surfaces 
in $\F$.
\end{lemm}
\begin{proof} We know that $E^4/(G,i)$ is singular along the image of points with 
non-generic isotropy; 
these are the surfaces from $\F$. (Note that all points whose isotropy group contains an element, which is not a symplectic reflection, is already a 2-torsion point).
A symplectic involution $T\in (G,i)$ is conjugate to $-T$, hence 
the four surfaces fixed by $T$ and
the four surfaces fixed by $-T$ map to the same four surfaces in 
$E^4/(G,i)$. 
It follows that $E^4/(G,i)$ is singular along the $60$ surfaces 
which are the images of the surfaces from $\F$.
\end{proof}

\begin{lemm}\label{remark-EPW} The ramification locus of the map $E^4/G\to E^4/(G,i)$ 
consists of~40 surfaces. 
They are contained in the set $\s$ of $60$ singular surfaces in $E^4/(G,i)$
and are characterized by the fact that they are not 
the images of the singular surfaces from $E^4/G$. 
\end{lemm} 
\begin{proof}
We saw in Proposition \ref{singE^4} that the images of~40 surfaces from $\F$ 
map to the singular locus of $E^4/G$. 
Since we can write these~40 surfaces explicitly, it is easy to check
that the actions of~$G$ and~$(G,i)$ are different on them.
It follows that the~20 singular surfaces of $E^4/G$ cannot be in the branch locus of $E^4/G\to E^4/(G,i)$. The remaining~40 singular surfaces of $E^4/G$ are actually fixed by an involution from~$(G,i)$.
\end{proof}


\begin{lemm}\label{lemm-D}
Each surface from $\s$ contains four points, 
these are the images of the points from $(E^4)^G$ on $E^4/(G,i)$. 
Each of the divisors $D+p \subset E^4$ for $p\in (E^4)^G$ contains 
fifteen sets of four points such that each such set of four points is contained 
in two of the fixed surfaces from $\F$.
\end{lemm}

\begin{proof}
From the proof of Proposition \ref{10 ODP} we know which
points from $(E^4)^G$ are contained in $D+p$.
Now it is a straightforward verification with Magma.
\end{proof}

\section{The morphism $E^4/G\to Y' \subset\PP^5$}\label{secEPW}
We find a line bundle on $E^4/G$ 
which gives a $2:1$ morphism to a sextic hypersurface $Y'\subset \PP^5$.
Then, in Section \ref{secDesing} we will show that $Y'=Y$ and give the proof of 
Theorem \ref{desingularisation}.

\subsection{The linear system $|\Delta|$ on $E^4/G$}
Let $D\subset E^4$ be  the $G$-invariant theta divisor as in Proposition \ref{10 ODP} 
and $L=\oo_{E^4}(D)$.
In this section we show that the image of the $G$-invariant theta divisor $D$ is not a Cartier divisor in $E^4/G$, 
but twice the image of $D$ defines a Cartier divisor $\Delta$ on $E^4/G$.
As $-1\in G$, diagonal multiplication by $i$ on $E^4$
induces an involution on $E^4/G$. 
The group $Aut(E^4,H)=U(H)$ induces an action of the symmetric group on
$\Sigma_6$ on $|\Delta|$ and this allows us to show that
$\Delta$ gives a morphism of degree $2$ that induces an isomorphism of $E^4/(G,i)$ with a sextic 
$Y'\subset\PP^5$, see Proposition \ref{main}.

By the Riemann--Roch theorem for Abelian varieties, $h^0(2D)=16$ and by the Theorem
of the Square, $2(D+p)\in |2D|$ for all $p\in (E^4)^G$ since $2p=0$.

Consider the morphism $\lambda$ given by the global sections of the invertible sheaf $L^2$.
Since $L^2$ is not a product polarization and is symmetric of type $(2,2,2,2)$,
we infer that $\lambda$ is a $2:1$ morphism equal to the quotient morphism $E^4\rightarrow E^4/(-1)$
and that $E^4/(-1)\subset\PP H^0(L^2)=\PP^{15}$.

\begin{lemm}\label{trivial_pi1}
Assume that $G \subset GL(n,\Z[i])$ is any finite group containing $-I$. 
Then the quotient $E^n/G$ has trivial fundamental group.
\end{lemm}

\begin{proof}
The method which we use to compute the fundamental group of this quotient is well-known and 
often applied in mathematical physics articles, which usually mention~\cite{physicsFundGroup} 
as the main reference for this topic.
It boils down to checking which elements of a certain extension of the action of $G$ 
to the action of $\Lambda_E^n\rtimes G$ on $\C^n$, where $\Lambda_E$ is a lattice such that 
$\C/\Lambda_E \simeq E$, have fixed points.
In particular, if there is an $A \in G$ such that $I-A$ has maximal rank, e.g. $A = -I$, 
then there are so many elements with fixed points that the fundamental group must be trivial.
\end{proof}

\begin{cor}
The symplectic desingularization $X_0$ of $E^4/G$ is simply connected:
$\pi_1(X_0) = 0$.
\end{cor}
\begin{proof}
By Lemma~\ref{trivial_pi1}, $\pi_1(E^4/G)$ is trivial, 
and by~\cite[Thm.~7.8 (a)]{kollar} the resolution does not change the fundamental group.
\end{proof}

In this section, $\eta$ is the quotient map $\eta:E^4\rightarrow E^4/G$. Recall that, by~\cite[Cor.~6.4]{DW}, the symplectic resolution $X_0$ of $E^4/G$ has $b_2(X_0) = 23$.

\begin{prop}\label{PicE^4/G}
The Picard group of $E^4/G$ has rank one and has no torsion.
The divisor $\eta(D)$ (with the reduced structure) is Weil but not Cartier whereas
$\Delta=2\eta(D)$ is Cartier, ample and generates the Picard group of $E^4/G$.
Moreover, $\Delta^4=12$ and $h^0(\Delta)=6$.
\end{prop}

\begin{proof}
Since $\pi_1(X_0) = 0$, then also $H^1(X_0,\oo_{X_0}) = 0$ and, from the exponential sequence,
$\Pic(X_0) \subseteq H^2(X_0,\Z)$. Then from the universal coefficient theorem we get an exact sequence
$$
0 \lra \Ext(H_1(X_0,\Z), \Z) \lra H^2(X_0,\Z) \lra \Hom(H_2(X_0,\Z), \Z) \lra 0.
$$
The triviality of $\pi_1(X_0) = 0$ implies that the first term of this sequence is 0.
Thus $\Pic(X_0) \subseteq \Hom(H_2(X_0,\Z), \Z)$, which is torsion-free.

It follows from \cite[Prop.~6.2]{DW} that the symplectic resolution $X_0\to E^4/G$ 
contracts at least $20$ independent divisors on $X_0$. 
On the other hand, since $b_2(X_0)=23$ and $h^{2,0}(X_0)=1$, 
the Picard rank of $X_0$ is at most $21$.
Thus the Picard group of $E^4/G$ has rank at most $1$.


We claim that $E^4/G$ is $2$-factorial. It is enough to prove this locally.
Since $E^4/G$ has only quotient singularities we see that the only non-factorial singularities 
it admits are isomorphic to the quotient singularity of $\C^4/G$.
In \cite[Lem.~2.10]{DW} it was shown that $\Cl(\C^4/G)=Ab(G)=(\Z_2)^4$, hence the claim follows.

It follows that $2\eta(D)$ is a Cartier divisor, necessarily ample from the Nakai-Moisheson criterion.
The pull-back of $\Delta$ on $E^4$ is a divisor from the system $|2D|$, so we can compute the 
self intersection of $(2\eta(D))^4$.
Consider the pull-back $\overline{\Delta}$ of $\Delta$ by the map $X_0\to E^4/G$.
Since $\Delta$ is ample we infer that $\overline{\Delta}$ is big and nef.
Observe that the proof of \cite[Prop.~2.1]{K} can be adapted for big and nef divisors. It follows that $\Delta^4=12 k^2$ for some $k\in \Z$.
As $(2D)^4=2^4(4!)=2^5\cdot 12$ and $\sharp G=2^5$ it follows that
 $\Delta^4=12$.
By \cite[Prop.~6.2]{DW} $b_2(X_0)=23$, thus we can repeat the arguments from \cite[Prop.~2.1]{K} 
to show that $h^0(\Delta)=h^0(\overline{\Delta}) =6$.

Since the self intersection of a big and nef divisor should be a multiple of $12$, 
we infer that $\Delta$ is the generator of the Picard group.
In order to see that $\eta(D)$ is not Cartier, it is enough to observe that otherwise the 
pull-back of $\eta(D)$ on the IHS fourfold $X_0$ would be a big and nef divisor with 
self intersection $\frac{12}{16}$.
\end{proof}

To understand the map from $E^4/G$ to $\PP^5$ provided by the linear system $|\Delta|$,
where $\Delta$ is the ample generator of the Picard group of $E^4/G$,
we study the pull-back of $|\Delta|$ to $E^4$.

\begin{prop}\label{prop-Di}
The linear system $|\Delta|$ on $E^4/G$ is generated by the $16$ divisors
$2\eta(D+p)$ for $p \in (E^4)^G$.
\end{prop}
\begin{proof}
As in the proof of Proposition  \ref{PicE^4/G}, these divisors $2\eta(D+p)$ are Cartier
and are invariant with respect to $G$, hence
they have self intersection $(2D)^4/\sharp G=12$.
Since we know that the Picard group of $E^4/G$ is $\ZZ\Delta$, with $\Delta^4=12$, 
we conclude that the divisors $2\eta(D+p)$ with $p \in (E^4)^G$ 
are in the linear system $|\Delta|$.

Notice that $\langle 2\eta(D+p):\;p \in (E^4)^G\rangle \subset |\Delta|\cong\PP^5$
and that the pull-back map $\eta^*$ maps $\langle 2\eta(D+p):\;p \in (E^4)^G\rangle$
linearly into the subsystem $\langle 2(D+p):\;p \in (E^4)^G\rangle \subset|2D|$ on $E^4$.
Proposition \ref{span >5} asserts that this subsystem of $|2D|$ has dimension at least $5$, 
hence also $\dim \langle 2\eta(D+p):\;p \in (E^4)^G\rangle=5=\dim |\Delta|$, 
which concludes the proof.
\end{proof}

Denote by $f\colon E^4/G\to  \PP^5$ the map given by $|\Delta|$.

\begin{prop}\label{prop-Dibpf}
The linear system $|\Delta|$ on $E^4/G$ is base point free
and the morphism $f\colon E^4/G\rightarrow\PP^5$ defined by $|\Delta|$ factors through the quotient map $h: E^4/G\rightarrow E^4/(G,i)$.
\end{prop}
\begin{proof} 
Using Proposition \ref{prop-Di}, we have to show that the subsystem 
$\langle 2(\eta(D+p)):\;p \in (E^4)^G\rangle$ 
is base point free on $E^4/G$.
Since the divisors $D+p$ for $p\in (E^4)^G$ are $G$-invariant, it is 
enough to prove that the linear subsystem $\langle 2(D+p):\;p \in (E^4)^G\rangle \subset |2D|$
is base point free on $E^4$.

We identify $H^0(2D):=H^0(E^4,\oo_{E^4}(2D))$ with the vector space of
rational functions on $E^4$ with poles of order at most two along $D$.
As $i(D)=D$, we have an endomorphism $i^*$ of $H^0(2D)$:
$$
H^0(2D):=\,\{f\in \CC(E^4):\,(f)+2D>0\,\},\qquad
i^*:\,H^0(2D)\longrightarrow\,H^0(2D),\quad
f\,\longmapsto f\circ i.
$$
As $i^2=-1$ and all functions $f$ in $H^0(2D)$ are even, so
$f(-x)=f(x)$, the map $i^*$ is an involution. 
Let $H^0(2D)=H^0_+\oplus H^0_-$ be the eigenspace decomposition into the even and the odd part.

Since $D$ is effective, we have the constant function $1\in H^0_+$.
Let $p\in (E^4)^G$, $p\ne 0$. 
Then $2(D+p)\in|2D|$, so there is a rational function $f_p\in H^0(2D)$
such that $(f_p)=2(D+p)-2D$. As $i(D+p)=D+p$, we have $i^*f_p=\pm f_p$ for some choice of sign.
From Corollary \ref{inters} we infer that there is a $q\in (E^4)^G$
which is not on $D$ and also not on $D_p$. Then $f_p$ has no pole 
and is not zero in $q$, so $f_p(q)\in\CC-\{0\}$. 
As $i(q)=q$, it follows that $(i^*f_p)(q)=f_p(i(q))=f_p(q)$, hence $i^*f_p=f_p$.
Therefore $f_p\in H^0_+$, for all $p\in (E^4)^G$, and the map defined by $|\Delta|$
factors through $E^4/(G,i)$.

Next, we show that $H^0_+$ is spanned by these $f_p$.
For each $p\in (E^4)^G$ there is a $G$-invariant divisor $D+q$ such that
$p\not\in D+q$. 
This implies that the fibers $\oo(2D)_p:=\oo(2D)/{\mathfrak{m}_p}\oo(2D)$ are generated 
by global sections in $H^0_+$ and one finds that the lift $i^*$ of $i$ to $\oo(2D)$ we consider
acts as $+1$ on the fibers over the fixed points $(E^4)^i=(E^4)^G$ of $i$.
The Atiyah-Bott Lefschetz Formula (\cite[Theorem 4.12]{AB}) states:
$$
\sum_{j=0}^4 Tr(i^*|H^j(E^4,\oo(2D))\,=\,\sum _{p\in (E^4)^i}\,
\frac{Tr(i^*_p:\oo(2D)_p\rightarrow \oo(2D)_p)}{\det(I-(\mbox{d}i)_p)}~.
$$
Since $H^j(\oo(2D))=0$ for $j>0$, $Tr(i^*_p)=+1$ for all $p$ and
$\det(I-(\mbox{d}i)_p)=(1-i)^4=-4$, we find $\dim H^0_+-\dim H^0_-=-4$, hence dim $H^0_+=6$ and
$\dim H^0_-=10$. Thus 
$$
\PP(H^0_+)\;=\;\langle 2(D+p):\;p \in (E^4)^G\rangle.
$$

By Wirtinger duality, the map $\lambda$ defined by $|2D|$ can be identified with the map
$$
\lambda'\colon\, E^4\ni x\,\longmapsto\, 
(D+x)+(D-x)\;\in \PP(H^0(E^4,\oo_{E^4}(2D))).
$$
In particular, $x\in E^4$ is a base point of the map defined by $\PP(H^0_+)$ exactly when
$(D+x)+(D-x)$ is the divisor of a global section $f_x$ with $f_x\in H^0_-$.
But then $i^*f_x=-f_x$ so $i^*(D+x)+i^*(D-x)=(D+x)+(D-x)$ and thus $x=\pm i(x)$, so $x=i(x)$ or
$x=i^3(x)$, hence $x\in (E^4)^i=(E^4)^G$. But we already found that $f_x\in H^0_+$ 
for such a fixed point. We conclude that there are no base points for the system $\PP H^0_+$
and thus also $|\Delta|$ is base point free.
\end{proof}

\begin{rem}  \label{diagram qoutient}
We have the following commutative diagram:
$$
\begin{CD} E^4 @> \eta>>  E^4/G@>h>>  E^4/(G,i)\\
@V |2D| VV @V|\Delta| VV @V  m VV \\
\PP^{15} @>\theta>> \;\;\PP^5 @>=>> \PP^5~, \end{CD}
$$ 

\

\noindent 
where $\theta$ is the linear projection from the $i^*$-eigenspace $\PP H^0_-$.

The action of $U(H)$ on $E^4$, which fixes the polarization and 
which commutes with the automorphism $i$ of $E^4$, 
induces an action of $U(H)$ on the projective space
$\PP H^0_+\cong |\Delta|$. 
As the subgroup $(G,i)$ of $U(H)$ acts trivially on this space, 
we get an induced action of the symmetric group $\Sigma_6=U(H)/(G,i)$ on $\PP^5$. 
\end{rem}

\subsection{The tangent cone}\label{section tgt cone}
First we make a study of the 
image of the tangent space to $E^4$ at the neutral element $0\in E^4$ under the quotient map $\eta\circ h :E^4 \rightarrow E^4/(G,i)=:Z'$. 
Since $\eta\circ h$ is a quotient map and the image point $p\in Z'$ of $0$ is singular on $Z'$,
the image of $T_0E^4$ via the induced quotient is the tangent cone $C_pZ'$  to $Z'$ at $p$. 

\begin{prop}\label{tgt cone}
The tangent cone to the image of $0$ to $Z'=E^4/(G,i)$ is isomorphic to the cone over the 
Igusa quartic ${\mathcal I}_4$ in $\PP^4$. It is defined by the two equations
$$
y_1+\ldots+y_6\,=\,0,\qquad
Ig_y:=(y_1^2+\ldots+y_6^2)^2\,-\,4(y_1^4+\ldots+y_6^4)\,=\,0,
$$
in $\CC^6$. 

The images of the tangent cones in $0$ of the ten $(G,i)$-invariant divisors with 
first Chern class $H$ and passing through $0$ are the ten hyperplane sections defined by
$$
y_a+y_b+y_c\,=\,y_d+y_e+y_f=0\qquad (\{a,b,c,d,e,f\}\,=\,\{1,2,\ldots,6\}).
$$
The group $U(H)$ acts on $C_pZ'$ through its quotient $U(H)/(G,i)\cong \Sigma_6$ 
as the standard representation i.e.~which simply permutes the six variables $y_1,\ldots,y_6$.

\end{prop}

\begin{proof}
The tangent cone to $Z'=E^4/(G,i)$ is
the spectrum of the ring of $(G,i)$-invariants on $T_0E^4$. Choosing a suitable basis of
$T_0E^4\cong\CC^4$, the action of $G$ on $T_0E$ is given by the simpler matrices from 
\cite[$\S$2.C]{DW}. The ring of invariants is generated by 
the following five polynomials in  4 variables: 
{\renewcommand{\arraystretch}{1.3}
$$
\begin{array}{lll}
p_0:=t_0^4+t_1^4+t_2^4+t_3^4,\quad&
p_1:=2(t_0^2t_1^2+t_2^2t_3^2),\quad&
p_2:=2(t_0^2t_2^2+t_1^2t_3^2),\quad\\
p_3:=2(t_0^2t_3^2+t_1^2t_2^2),\quad&
p_4:=4t_0t_1t_2t_3,
\end{array}
$$
}
so $\CC[t_0,\ldots,t_3]^{(G,i)}\cong \CC[p_0,\ldots,p_4]$. 
Introducing the polynomial ring in 5 variables $\CC[P_0,\ldots,P_4]$, the homomorphism 
$\CC[P_0,\ldots,P_4]\rightarrow\CC[t_0,\ldots,t_3]^{(G,i)}$ 
which sends $P_i\mapsto p_i$  is surjective, and its kernel 
is generated by the quartic polynomial
$$
Ig_P:=P_1^2P_2^2+P_1^2P_3^2+P_2^2P_3^2\,+\,
(-P_0^2+P_1^2+P_2^2+P_3^2-P_4^2)P_4^2\,-\,2P_0P_1P_2P_3.
$$
Hence $C_pZ'\cong Spec(\CC[P_0,\ldots,P_4]/(Ig_4))$, more concretely, $C_pZ'$ is  
isomorphic to the image of $\CC^4$ in $\CC^5$ by the map defined by the five $p_i$. 

The zero locus of the quartic polynomial $Ig_P$ in $\PP^4$ is known as the Igusa quartic 
${\mathcal I}_4$. The group $U(H)/(G,i)\cong \Sigma_6$ acts on ${\mathcal I}_4$.
To make this action visible, we define the following six linear combinations of the $p_i$:
$$
y_1:=p_0+3p_4,\quad y_2:=p_0-3p_4,\quad 
y_3:=\mbox{$\frac{1}{2}$}(-p_0+3p_1+3p_2+3p_3),
$$
$$
y_4:=\mbox{$\frac{1}{2}$}(-p_0+3p_1-3p_2-3p_3),\quad
y_5:=\mbox{$\frac{1}{2}$}(-p_0-3p_1+3p_2-3p_3),\quad
y_6:=\mbox{$\frac{1}{2}$}(-p_0-3p_1-3p_2+3p_3).
$$
One easily verifies that
$$
y_1+\ldots+y_6\,=\,0.\qquad\mbox{Let}\quad
Ig_y:=(y_1^2+\ldots+y_6^2)^2\,-\,4(y_1^4+\ldots+y_6^4),
$$
then, after replacing $p_j$ by $P_j$ in the definition of the $y_i$, the polynomial $Ig_y$ 
becomes the polynomial $Ig_P$ up to a scalar multiple.
Thus $\Sigma_6$ acts on ${\mathcal I}_4$ by permuting the coordinates $y_j$ 
(and explicit computations show that this is indeed the action induced by $U(H)$).

There are $10$ $(G,i)$-invariant divisors which contain $0$, and $0$ is a node on these divisors. 
The tangent cone to such a divisor is then a $(G,i)$-invariant quadratic hypersurface in $T_0E^4$. 
There are exactly 10 such hypersurfaces and they are permuted by the action of $U(H)$, 
explicit equations can be found in \cite[(3.13)]{DW}.
It is easy to check that the image of such a quadric in $\CC^5$ is the intersection of
${\mathcal I}_4$ with one of the following  hyperplanes:
$$
y_a+y_b+y_c\,=\,y_d+y_e+y_f=0\qquad (\{a,b,c,d,e,f\}\,=\,\{1,2\ldots,6\}).
$$
As $\sum y_i=0$, the equation $y_a+y_b+y_c=0$ implies the equation $y_d+y_e+y_f=0$.
For example, $y_0+y_1+y_2=0$ has preimage defined by $p_0 + p_1 + p_2 + p_3=0$
and this is  $(t_0^2+t_1^2+t_2^2+t_3^2)^2$, which is (up to scalar multiple and 
replacing $t_j$ by $x_j$) the quadratic form $\phi_{14}$ in \cite[(3.13)]{DW}.
The fact that we find the squares of the quadratic forms corresponds to the fact that these 
ten hyperplanes are tangent to the Igusa quartic. 
\end{proof}
Recall that $Y'$ is the image of the map $f: E^4/G\to \mathbb{P}^5$ defined by the system $|\Delta|$.
\begin{cor} The tangent cone to $Y'$ at the image of $[0]$ is a cone over the Igusa quartic $I_4$.
\end{cor}
\begin{proof} 
If $D+p$ is a $(G,i)$-invariant divisor passing through $0$ on $E^4$, 
then $2(D+p)\in |2D|$ by the Theorem of the Square and this divisor lies in the pullback by the quotient map of the five dimensional 
linear system defining the map $f:E^4/G\rightarrow\PP^5$. It is easy to check that the squares 
of the ten quadratic forms defining the tangent cones to these divisors span the space spanned 
by the generators $p_0,\ldots,p_4$ of the $(G,i)$-invariants on $T_0E^4$.
Hence the tangent cone
to $Z'=E^4/(G,i)$ in the image of $0$ is embedded in $\PP^5$ by the map $m: E/(G,i) \to \mathbb{P}^5$ induced by $f$ (see Proposition \ref{prop-Dibpf} ) 
i.e. the tangent cone to $Y'$ is a cone over the Igusa quartic $I_4$.

\end{proof}

We infer also that the considered action of $\Sigma_6$ on $\PP^5$ 
that preserves $Y'$ is given by the permutation of coordinates:

\begin{cor}\label{sigma} The action of the group $\Sigma_6=U(H)/(G,i)$ on 
$\mathbb{P}^5=\mathbb{P}(|\Delta|)$ induced by the action of $U(H)$ on $E^4$ 
(see \ref{diagram qoutient}) is the action by permutation of coordinates.
\end{cor}
\begin{proof}
We use the fact that $C_pZ'$ spans $\CC^5$,
that $\CC^5$ is dense in $\PP^5$ and the description from Proposition \ref{tgt cone}
of the action of  $\Sigma_6$ on $C_pZ'$.
\end{proof}

\begin{rem} One can use the action of $\Sigma_6$ to provide the following approach to the proof that $Y'=Y$ that is a possible alternative to the proof we give in Section 6.
From the explicit description of 
the generators of $U(H)$ and of the sixteen $G$-fixed points in $E^4$, one finds that 
$(E^4)^G$ consists of two $U(H)$-orbits, one is $\{0\}$ and the other has $15$ elements. 
For $q\in (E^4)^G$, $q\neq 0$, the line spanned by $p=f(0)$ and $f(q)$ in $\PP^5$ 
intersects $\CC^5\subset\PP^5$ in a linear subspace of dimension one. 
The fifteen lines in $\CC^5$ we obtain in this way form one $\Sigma_6$-orbit. 
This suffices to identify these lines and with some additional effort, 
we can then find the images of the remaining fifteen points in $(E^4)^G$. 
Using the fact that these map to  
singular points on $Y$, with quartic tangent cones, 
and the description of the images of the $(G,i)$-invariant divisors from Proposition \ref{tgt cone}, 
we would then show that 
$Y'=Y$ with $Y$ as in Proposition \ref{equationY}.
It is interesting to look at the above action on the Igusa quartic  from the point of view of \cite[p.~254]{P}.
\end{rem}

Let us now pass to the proof of our main result in this section: 

\begin{prop}\label{main} The image $Y'$ of the morphism $f:E^4/G\rightarrow\PP^5$
is a sextic hypersurface $Y' \subset \PP^5$. 
Moreover, the morphism $f$ is the quotient map of the involution induced by $i$ acting on $E^4/G$,
hence $Y'\cong E^4/(G,i)$. 
\end{prop}

\begin{proof}
We saw that $\Delta$ is ample, $\Delta^4=12$ and $|\Delta|$ is a base point free linear system.
First, it follows that $\dim \varphi_{|\Delta|}(E^4/G)\subset \PP^5=4$ and 
the degree of this image divides $12$.
By Lemma \ref{automorphism i} the morphism $f\colon E^4/G\to \PP^5 $ factorizes through the quotient by the involution $i$. 
We infer a factorization of the map $f$:
\[ 
f:\, E^4/G\,\stackrel{h}{\longrightarrow}\, E^4 /(G,i)\,\stackrel{m}{\longrightarrow}
\,Y'\,\subset\,\PP^5.
\]
Our aim is first to show that the image $Y'$ of $m$ is a sextic, 
and next that $m$ is an isomorphism.

Since the tangent cone $C_pY'$ is the image of $C_pZ'$ by the embedding induced by $m$ it has degree four, by Proposition \ref{tgt cone}. It follows that the degree of $Y'\subset\PP^5$ is higher than $4$ (Clearly $Y'\neq C_pY'$).
Since $\Delta^4=12$ and $f$ has degree at least two, 
we infer that $Y'$ is a sextic hypersurface in $\PP^5$. 

It follows from the adjunction formula for the birational morphism $m$ 
that the sextic $Y'\subset \PP^5$ is normal.
Indeed, we have $\omega_{E^4/(G,i)}=m^{\ast}(\omega_{Y'})-C$, 
where $C$ is the conductor divisor supported on the non-normal locus of $Y'$.
Since $Y'$ is a sextic, $\omega_{Y'}$ is trivial, and thus $C$ is the zero divisor and $Y'$ is normal.
This implies also that $m$ is an isomorphism because $\Delta$ is ample and base point free,
so $f$ cannot contract any curve.
\end{proof}

\begin{rem} 
The double cover $E^4/G\to E^4/(G,i)$ is determined by the sheaf 
$f_{\ast} \oo_{E^4/G}=\oo_{E^4/(G,i)}\oplus \G$
such that $E^4/G=\Spec_{E^4/(G,i)}( \G\oplus \oo_{E^4/(G,i)})$. It can be shown that $\G$ is a symmetric
sheaf such that $\G(3)$ is globally generated and fits in the exact sequence
$$ 
0\to \Omega_{\PP^5}^3(3)\to W\otimes \oo_{\PP^5}\to \G(3)\to 0~.
$$
This gives another proof (without using section \ref{section epw sextic}) 
that $Y'\subset \PP^5$ is an EPW sextic (see \cite{EPW}).
Note that the sheaf $\G$ can be seen as a kind of Casnati-Catanese sheaf, 
however it has a complicated local structure around the
$16$ most singular points of $E^4/(G,i)$. 

\end{rem}

We now obtain some further information on the geometry of $Y'$ which will be used 
in the next section to prove that $Y'=Y$.

\begin{cor} 
The sextic ${Y'}\subset \PP^5$ is singular along $60$ planes.
\end{cor}

\begin{proof}
It follows from Lemma \ref{singular set} that the fourfold $E^4/(G,i)$ 
is singular along $60$ surfaces.

Note that all the sixty singular surfaces of $E^4/(G,i)$ are isomorphic to each other 
(since the group generated by $(E^4)^G$ (acting by translations) and $U(H)$ 
acts transitively on the corresponding $120$ surfaces in $E^4$).
On the other hand it is known that a surface section of ${Y'}\subset \PP^5$ 
admits no more than 65 nodes (recall that $E^4/(G,i)$ has transversal $A_1$ singularities along at a generic point on the singular surfaces).
It follows that all the singular surfaces of ${Y'}\subset \PP^5$ are planes. 
\end{proof}

Recall that there exists a unique, up to projective isomorphism, 
normal cubic hypersurface in $\PP^4$ with $10$ isolated singularities 
(that have to be ordinary double points).
It is called the Segre cubic. 
See \cite{Dolgachev} for many beautiful classical facts about this threefold.
The Segre cubic in $\PP^5(x_0,\dots,x_5)$ can be defined  by the following equations:
\begin{equation}
\label{equation1} x_0+\ldots+x_5\,=\,0,\qquad x_0^3+\ldots+x_5^3\,=\,0~. 
\end{equation}

The action of the permutation group $\Sigma_6$ that permutes the variables on $\PP^5$ 
preserves the above cubic.
It is singular at the points in the orbit of the point $(1,1,1,-1,-1,-1)$ under the action of $\Sigma_6$.
This cubic contains also $15$ planes in the $\Sigma_6$ orbit of the plane defined by 
the equations $x_0+x_1=0$, $x_2+x_3=0$, $x_4+x_5=0$.
There are exactly fifteen hyperplanes cutting the cubic along the sum of three planes, 
these hyperplanes are defined by the equations $x_i+x_j=0$  for $0\leq i< j \leq 5$.

\begin{cor}\label{Y'geometry}
The sextic ${Y'}\subset \PP^5$ is tangent to sixteen hyperplanes along Segre cubics 
such that the singular points of the cubic are the points from the set $f(\eta((E^4)^G))=\R$. 
The sextic ${Y'}\subset \PP^5$ is singular along the $15$ planes contained in
each of this cubic. 
Moreover, the intersection of ${Y'}\subset \PP^5$ with two of the hyperplanes defined by the divisors 
$D+p$ for $p\in (E^4)^G-{0}$ is a union of three planes. Each singular plane in $Y'\subset \PP^5$
is contained in four tangent hyperplanes.
\end{cor}

\begin{proof}  
We shall show that each divisor $D+p$ for $p\in (E^4)^G$ maps to a cubic that is singular at
ten isolated points. It is known that only the Segre cubic has this property.

It is enough to give a proof of our statements for $D$.
From Proposition \ref{PicE^4/G} we infer that the image of the divisor 
$D$ in $E^4/(G,i)\subset \PP^5$ is contained in 
a hyperplane $K\subset \PP^5$ that is tangent to the sextic.

The $120$ fixed surfaces from $\F$ map to $60$  singular surfaces in $E^4/(G,i)$, 
hence they are the singular planes on $Y'\subset \PP^5$.
Since $D$ contains exactly $10$ of the $16$ points in $(E^4)^G$, 
its image $K\cap Y$ contains ten of the $16$ points from $\R$.
Each fixed surface contains four fixed points from the set $\R$ and we claim that 
the images of these points on $Y'\subset \PP^5$ are non-collinear 
(so a singular plane on $Y'$ is spanned by the four points of $\R$ contained in it).

In fact, suppose that the images of the four points of $\R$ in a plane are collinear.  
Choosing points $p$ and $q$ from these four points, 
we can find another fixed surface 
from the set $\F$ containing only these two of the four points.
The corresponding planes cut along a line spanned by the image of $p$ and $q$.
Now this line has to contain the remaining two points, this is a contradiction 
with our choice of the other fixed surface, so the claim follows.

It follows that the thirty surfaces considered in Lemma \ref{lemm-D}, containing four points 
from the ten contained in $\R_0=K\cap \R$, map to fifteen planes, each spanned by the images of those points, 
hence these thirty surfaces are contained in $K$. 

As $D$ is $(G,i)$-invariant, we see that the reflections generating this group 
act on $D$ in such a way that they fix thirty surfaces. 
It follows that the image of $D$ can only be singular at isolated points, 
i.e.~at the images of 
singular points of $D$. 
Since through each point in $(E^4)^G$ there are two fixed surfaces from $\F$ contained in $D$ 
intersecting only at this point, these surfaces map to two planes in $K$ intersecting 
only at a point from $\R_0$, so this point must be singular on the cubic threefold. 
We deduce that the cubic is (only) singular at the 
ten points in $\R_0$ and hence it must be the Segre cubic.

It is known that there are exactly $15$ planes contained in the Segre cubic. 
On the other hand, by Lemma \ref{lemm-D}, the thirty surfaces from $\F$ that are contained in $D$
map to fifteen planes contained in the cubic, hence the planes in the cubic are
the images of these surfaces.

The intersection of two $(G,i)$-invariant divisors maps to the intersection of two tangent hyperplanes 
to $Y'\subset \PP^5$. In particular, it maps to the intersection of a Segre cubic threefold with a hyperplane, 
hence it consists of at most three planes. 
Thus two $(G,i)$-invariant divisors cut each other along at most six surfaces from $\F$.
On the other hand, given two such divisors, we easily find, using Proposition \ref{10 ODP},  
three sets of four points contained in a given surface from $\F$, 
each one contained in both of these divisors.
We show similarly that any fixed surface is contained in four $(G,i)$-invariant divisors.
\end{proof}

\begin{rem} From the incidence of the $120$ fixed surfaces from $\F$ we deduce that
each plane in the singular locus of $Y'\subset \PP^5$ cuts $12$ of the remaining planes from this locus 
along six lines (such that three planes passes through one line).  
\end{rem}

\section{The proof that $Y=Y'$}\label{secDesing}
We proved in Corollary \ref{sigma} that the image $Y'=f(E^4/G)\subset \PP^5$ is invariant under the action of $\Sigma_6$ by the permutation of 
coordinates. Moreover, from Corollary \ref{Y'geometry} it is tangent to $16$ hyperplanes.
In this section we show that such a $\Sigma_6$-invariant sextic~$Y'\subset \PP^5$ 
can be easily reconstructed from the $\Sigma_6$-invariant set of~$16$ 
hyperplanes tangent to it. 
This allows us to show that $Y=Y'$. 
Then, in Section~\ref{section_proof_desing}, we prove Theorem~\ref{desingularisation}.

\subsection{The equation of the sextic}\label{section_sextic_equation}
We start by classifying sets of 16 hyperplanes which are invariant under the action of $\Sigma_6$
which acts by permutations of the coordinates on $\PP^5$.
Let $t \in \CC$ and $0 \leq i, j, k \leq 5$ be distinct indices. 
We consider (families of) hyperplanes:
\begin{itemize}
\item $H$ defined by $x_0+\dots+x_5=0$,
\item $H^t_i$ defined by $x_i+t(x_0+\dots+x_5)=0$,
\item $H^t_{i,j}$ defined by $x_i+x_j+t(x_0+\dots+x_5)=0$,
\item $H_{i,j,k}$ defined by $x_i+x_j+x_k-\frac{1}{2}(x_0+\dots+x_5)=0$.
\end{itemize}

\begin{lemm}\label{lemma_16_planes_families}
There are exactly two one-parameter families of $\Sigma_6$-invariant sets of 16 hyperplanes with the 
following property: every such a set determines $60$ planes, such that each plane is contained in four 
of these hyperplanes (this is one of the properties of $Y'$ from Prop.~\ref{Y'geometry}). They are:
$$
\mathcal{H}^t_1\,=\{H\} \cup \{ H^t_{i,j} \colon 0\leq i,j \leq 5\}
$$
and 
$$
\mathcal{H}_2^t\, = \{ H_{i,j,k} \colon 0\leq i,j,k \leq 5\} \cup \{ H^t_i \colon 0\leq i \leq 5\}.
$$
\end{lemm}

\begin{proof} Consider the action of $\Sigma_6$ on a hyperplane with equation 
$a_0x_0+\dots+a_5x_5=0$. If its orbit has $\leq 16$ elements,
then the coefficients $\{a_0,\dots,a_5\}$ 
can take only two different values. That is, it must be one of 
$H$, $H^t_i$, $H^t_{i,j}$, $H_{i,j,k}$ for some $i,j,k$ and $t$. 
The lengths of $\Sigma_6$-orbits of hyperplanes of these types are, respectively, $1$, $6$, $15$ and $10$. 
Thus, to obtain an invariant set of cardinality $16$, we have two possibilities: 
to take the union of the orbits of $1$ and $15$ elements or of the  orbits of $6$ and $10$ elements. 

In both cases it is easy to check that for any $t \in \CC$ there are the required sets of $60$ planes,
they are the intersection of the following sets of four hyperplanes
(the indices for each set are different):

\begin{itemize}
\item[(1a)] $H^t_{i_1,i_2}$, $H^t_{i_2,i_3}$, $H^t_{i_3,i_4}$, $H^t_{i_4,i_1}$, 
there are $45$ planes of this type,
\item[(1b)] $H$, $H^t_{i_1,j_j}$, $H^t_{i_2,j_2}$, $H^t_{i_3,j_3}$, there are $15$ such planes.
\end{itemize}

\begin{itemize}
\item[(2a)] $H_{i,j_1,j_2}$, $H_{i,j_2,j_3}$, $H_{i,j_3,j_4}$, $H_{i,j_4,j_0}$, 
there are $15$ such planes,
\item[(2b)] $H_{i,j,k_1}$, $H_{i,j,k_2}$, $H^t_{k_1}$, $H^t_{k_2}$, 
there are $45$ such planes.
\end{itemize}
\end{proof}

We want to find all polynomials $f \in \CC[x_0,\ldots,x_5]$ such that the corresponding 
sextic hypersurface $Y_{f}$ satisfies the following:
\begin{itemize}
\item $Y_f$ is invariant with respect to the $\Sigma_6$-action by 
permutations of coordinates,
\item the sixteen hyperplanes from the configuration $\mathcal{H}_1^t$ or $\mathcal{H}_2^t$, 
for some  $t\in \CC$, are tangent to $Y_{f}$ along cubics,
\item if the intersection of four of these sixteen
hyperplanes is a plane, then this plane is contained in the singular locus of $Y_f$.
\end{itemize}

The fact that $Y_{f}$ is $\Sigma_6$-invariant means that $f$ is symmetric 
or anti-symmetric under the action of $\Sigma_6$. In any case, $f$ is invariant
with respect to the alternating group. But the invariants of the alternating group are
generated by those of the symmetric group and the polynomial $\prod_{i<j}(x_i-x_j)$,
which has degree $15$. Hence $f$ must be a symmetric polynomial
and thus $f$ is a linear combination of the symmetric polynomials:
$P_{j_1\ldots j_k}=\sum_{i_1,\ldots,i_k}x_{i_1}^{j_1}\cdots x_{i_k}^{j_k}$ such that 
$i_1,\dots, i_k\in \{0,\ldots,5\}$ are pairwise different and 
$j_1\geq \dots \geq j_k$ with $j_1+\dots+j_k=6$. 
Thus $f$ is an element of an $11$-dimensional vector space of polynomials. 
Now we shall determine all possible sets of coefficients in 
$$
f(x_0,\ldots,x_5) = \sum_{j_1+\ldots +j_k =6} a_{j_1\ldots j_k} P_{j_1\ldots j_k}(x_0,\ldots,x_5).
$$ 
The following lemma, that is easy to verify, shows that in all but some exceptional cases
it suffices to consider the case $t=-1/2$.

\begin{lemm}
We define linear maps on $\CC^6$ which commute with the $\Sigma_6$-action by
$$
N_1^t\,:=\,-(t+1)M_6 + (6t+2)\id_6,\qquad N_2^t\,:=\,-(t+1)M_6 + (6t+1)\id_6,
$$
where $M_6$ is the $6\times 6$ matrix with all entries equal to $1$ and $\id_6$ is the
$6\times 6$ identity matrix.

Then $N_i^t$ induces an isomorphism on $\PP^5$ that maps 
$\mathcal{H}_i^t$ to $\mathcal{H}_i^{-1}$ for $i=1,2$, unless $i=1$, $t=-1/3$ or $i=2$, $t=-1/6$.
\end{lemm}

We consider the restrictions on the coefficients $a_{j_1\ldots j_k}$ 
coming from the assumption that~$f$ 
and all its partial derivatives vanish along one plane of type (1a) and (1b) or 
respectively one of type (2a) and one of type (2b).
From the symmetry of~$f$ it follows that $Y_{f}$ is then also singular along all the 
$60$ planes. 
To find all sequences $(a_{j_1\ldots j_k})$ satisfying these conditions, 
one just has to compute 
the kernel of the matrix whose entries are $P_{j_1 \ldots j_k}$ and its partial derivatives
restricted to the two chosen planes.
We obtain a unique solution, up to scalar multiple.
Notice that in the second case we did 
find the polynomial defining $Y$ from Proposition \ref{equationY}: 

\begin{prop}\label{sextic}
For the first case, there is a unique sextic
$$
Y^\vee\,=\,Y_{F_6^\vee}:\quad F_6^{\vee}=P_6 - P_{42} + 2P_{222} +16P_{111111}~
$$
which satisfies the conditions for the sixteen hyperplanes $\mathcal{H}^t_1$ with $t=-1/2$.

For the second case, there is a unique sextic
$$
Y\,=\,Y_{F_6}:\quad F_6=P_6 - P_{42} + 2P_{222} -16P_{111111}~
$$
which satisfies the conditions for the  sixteen hyperplanes $\mathcal{H}^t_2$ with $t=-1/2$.
\end{prop}

\begin{cor}\label{cor_family1_epw}
The hypersurfaces $Y^\vee\subset \PP^5$ and $Y\subset \PP^5$ are isomorphic.
\end{cor}
\begin{proof}
The only thing to observe is that $F_6(-x_0,x_1,\ldots,x_5)=F_6^\vee(x_0,x_1,\ldots,x_5)$,
so changing the sign of an odd number of variables interchanges the two cases.
\end{proof}

\begin{cor}\label{y=y'}
The EPW sextic hypersurfaces $Y\subset \PP^5$ (the image of $S^{[2]}$) and $Y'\subset \PP^5$ (the image of $E^4/G$)
are isomorphic.
\end{cor}
\begin{proof} By Theorem \ref{equationY} it is enough to prove that $Y'\subset \PP^5$ is defined by equation $F_6$ or $F_6^{\vee}$. 
It follows from Proposition \ref{Y'geometry} that the equation defining $Y'\subset \PP^5$ satisfy the conditions satisfied by $F_6$ and $F_6^{\vee}$. We conclude by Proposition \ref{sextic}.
\end{proof}

\begin{rem} \label{isotrivial}
For $t=-1/3$ the corresponding sextic is 
the square of a 
cubic which is singular in 10 lines. 
For $t=-1/6$ the corresponding sextic is singular along 60 planes such that 
they all intersect in one point. 
This gives us isotrivial degenerations of our EPW sextic $Y$.
\end{rem}

It turns out that $Y^{\vee}$ and $Y$ are related in one more way.

\begin{prop}
The sextics $Y^{\vee}$ and $Y$ in $\PP^5$ are projectively dual 
to each other.
\end{prop}

\begin{proof}
Subsituting the gradient $(\ldots,\partial F_6/\partial x_i,\ldots)$ of the 
equation defining $Y$ in the polynomial $F_6^\vee$ defining $Y^\vee$,
one finds, using e.g.\  Macaulay2~\cite{M2},
the product of $F_6$ with another polynomial. Hence the dual of $Y$
is $Y^\vee$.
In particular, the 16 hyperplanes are mapped to the 16 points with singularity 
$\CC^4/(G,i)$, 
and the Segre cubics in these hyperplanes are contracted to points.
(Notice that the Segre cubic and the Igusa quartic are projectively dual threefolds 
in $\PP^4$.)
\end{proof} 

\begin{rem} \label{ime4g}
The images of the $16$ points in $(E^4)^G$ in $Y'\cong Y$ are the points in the orbit
of the point $p_0:=(1:1:\ldots:1)$ under the action of the group which changes an even number 
of signs (this action is induced by the action of $(E^4)^G$ on $E^4$ by translation)
and the group $\Sigma_6$ (which is induced by the action of $U(H)$ on $E^4$).
In fact, these points are the singular points on the Segre cubics that are tangent 
hyperplane sections of $Y$. The point $p_0$ was also identified with the image 
of a surface in $S^{[2]}$ in the proof of Proposition \ref{equationY}, 
see also 
Remark \ref{remEx}. In particular these 16 points in $Y$ are the set $Y_A[4]$ where
$Y=Y_A$, see Remark~\ref{YA4}.
\end{rem}

To describe the incident planes,
we need the following combinatorial description of the 60 singular planes of $Y'$:
\begin{rem}\label{remark_relation_plane_types}
A partition $\{\{i_1,j_1\}, \{i_2,j_2\}, \{i_3,j_3\}\}$ of $\{0,\ldots,5\}$, defining a plane of type (1b), determines in a natural way three sequences $(i_1,i_2,j_1,j_2)$, $(i_1,i_3,j_1,j_3)$, $(i_2,i_3,j_2,j_3)$, which correspond to planes of type (1a). Note that the orders of pairs and indices in pairs do not matter, i.e. if we change them, we still get a sequence determining the same plane. This way we obtain a natural subdivision of the set of 60 planes into subsets of cardinality~4: each consists of a plane of type (1b) given by a partition of $\{0,\ldots,5\}$ and 3 planes of type (1a) described by sequences determined by this partition. We show below that the set of 20 incident planes consists of~5 such subsets, cf. Remark~\ref{rem_planes_intersection} and section~\ref{section_geom_quot}.
\end{rem}

The following is a nice exercise:

\begin{lemm}\label{lemma_intersection_types}
The pairs of planes intersecting in a point are as follows.
\begin{itemize}
\item Two planes of type (1b) intersect in a point if and only if the corresponding partitions do not have any common component.
\item Two planes of type (1a) intersect in a point if and only if the corresponding sequences come from the same partition (as described in Remark~\ref{remark_relation_plane_types}) or from two partitions with empty intersection.
\item A plane of type (1a) intersects a plane of type (1b) in a point if and only if the sequence representing the first plane comes from a partition which represents the second one or which has empty intersection with this partition.
\end{itemize}
\end{lemm}
\begin{prop}\label{choices}
There are exactly 6 possible choices of the set of 20 incident planes in the sextic~$Y^{\vee}_t$.
 They are all obtained from the one in Proposition \ref{20inc} 
by the $\Sigma_6$-action. The stabilizer of such a configuration is 
isomorphic to $\Sigma_5$.
\end{prop}
\begin{proof}
Assume first that there are at least 16 planes of type (1a) in such a set. 
Then their corresponding sequences must come from at least 6 different partitions. 
But by Lemma~\ref{lemma_intersection_types} some two of these partitions must have 
a common element, so we would have two planes which intersect along a line or 
do not intersect at all. If they intersect along a line, then by~\cite[Prop.~2.2]{ogrady-incident} this configuration is contained in an infinite family of incident planes. Hence we may restrict to the case where they intersect in a point.

Hence we may assume that there are 5 planes of type (1b) in the chosen set. 
By Lemma~\ref{lemma_intersection_types} the corresponding partitions do not have 
a common component, thus their union consists of all possible pairs of indices. 
Hence there is no other plane of type (1b). 
Again by Lemma~\ref{lemma_intersection_types}, the only possible planes of type (1a) 
which can appear in the set are those represented by sequences which come from 
partitions corresponding to the chosen planes of type (1b). 
There are 15 of them, so the configuration can be completed in a unique way.
\end{proof}

\begin{rem}
It is worth noticing that six possible choices of the set of 20 incident planes in~$Y$ correspond to six possible choices of a 32-element subgroup isomorphic to $G$ inside $(G,i)$.
\end{rem}

We are ready for the proof of our Theorem.

\subsection{Proof of Theorem \ref{desingularisation}}\label{section_proof_desing}
Let $\overline {S^{[2]}}\to Z\to Y$ be the Stein factorization of the 2:1 morphism 
$\overline{g}$ constructed above.
We saw also that there exists also a finite 2:1 morphism $E^4/G \to Y$.
Our aim is to show that $Z$ is isomorphic to $E^4/G$ by proving that the two double covers 
$Z\to Y \leftarrow E^4/G$ are the same. 
First we shall show that the ramification loci of the morphisms are the same.

The sextic $Y$ is singular along 60 planes. 
In Lemma \ref{remark-EPW} we already identified forty of them
that are in the ramification locus of $E^4/G\to Y$.
We already showed in Corollary \ref{branch} that the ramification locus $Z \to Y$ 
also consists of $40$ planes.
From Proposition \ref{singZ} the remaining twenty singular planes of $Y$ 
are the images of the singular surfaces on $Z$ that are incident. 
It follows that the images of the above singular surfaces from $Z$ 
are incident planes on $Y\subset \PP^5$.
From Proposition \ref{choices} we infer that the choices of the twenty incident planes differ
by a projective transformation fixing $Y\subset \PP^5$. 
Thus the ramification loci of the maps $Z\to Y \leftarrow E^4/G$ are the same.

Finally, consider the coverings $Z\to Y$  and $E^4/G\to Y$.
The two maps have the same ramification locus, moreover outside the singular locus 
of $Y$ both maps are \'etale covers. 
Since $\overline{S^{[2]}}$ is simply connected we infer that the fundamental group 
$\pi_1(Y-Sing(Y))=\Z_2$.
From the uniqueness of integral closures this is enough to conclude that $Z$ is isomorphic to 
$E^4/G$ (they are both universal covers in codimension 1).
\begin{rem} In the proof of Theorem \ref{desingularisation} (in Prop.~\ref{PicE^4/G}) we use the results from \cite{DW} about the existence of a symplectic desingularisation $X_0\to E^4/G$.
It is an interesting problem to compare the manifolds $X_0$ and $\overline{S^{[2]}}$.
\end{rem}
\subsection{Final Remarks}
It follows from Proposition \ref{main} that $X_0$ is of $K3^{[2]}$-type as an double cover of an EPW sextic. Knowing this we have a direct, lattice theoretical, 
proof that $X_0$ is birationally isomorphic to $S^{[2]}$.
This result is weaker than Theorem \ref{desingularisation}, 
but the proof is much shorter! 

Recall that the second integral cohomology group of a $K3^{[2]}$-type IHS fourfold, 
with the Beauville-Bogolomov form, is isomorphic to the lattice $\Gamma$ 
that is an orthogonal direct sum
\begin{equation}\label{isoh2}
\Gamma\,:=\,\Lambda_{K3}\,\oplus\,\ZZ\xi,\qquad \Lambda_{K3}\,\cong\,E_8(-1)^2\oplus U^3,\quad
\xi^2=-2.
\end{equation}
The following result was shown to us by G.\ Mongardi.

\begin{prop}\label{mongardi}
The IHS 4-folds $X_0$, the desingularization of $E^4/G$, and $S^{[2]}$ are
birationally isomorphic.
\end{prop}

\begin{proof}
From the construction of $X_0$ we know that 
the 23-dimensional vector space $H^2(X,\QQ)$ has a 21 dimensional subspace spanned by 
the class of the divisor $\Delta$ and $20$ exceptional divisors which map to the singular
surfaces in $E^4/G$. Each class in this subspace is invariant under the action of $i^*$, 
where $i\in Aut(X_0)$ is the covering involution for the map $X_0\rightarrow Y$.
As the holomorphic two-form on $X_0$ does not descend to $Y$, we see that $i^*=-1$ on 
a complementary two-dimensional subspace. Thus $H^2(X_0,\ZZ)$ contains, with finite index, 
the direct sum of the $i^*$-invariant and anti-invariant sublattices, which are the
Picard group of $X_0$ and the transcendental lattice $T$ respectively.
The lattice $H^2(X_0,\ZZ)\cong \Gamma$ is not unimodular, but we can embed it in
an even unimodular lattice as follows. 
Let $\tilde{\Gamma}\subset \Gamma\otimes_\ZZ\QQ\oplus \QQ\eta$,
with $\eta^2=2$, be the lattice generated by $\Gamma$ and $e_1:=(\xi+\eta)/2$.
Let $e_2:=e_1-\xi\in \tilde{\Gamma}$, then $e_1,e_2$ generate a hyperbolic plane $U$
(so $e_1^2=e_2^2=0$ and $e_1e_2=1$) and $\tilde{\Gamma}=\Lambda_{K3}\oplus U$.

Since the discriminant group of $\Gamma$ is $\ZZ_2$ 
and $i^*$ acts trivially on it, it extends to an isometry $j$  
of $\tilde{\Gamma}$ with $j(\eta)=\eta$.
Then the sublattice of $j$-anti-invariants
in $\tilde{\Gamma}$ has rank two and is isometric to $T$. 
As $j$ is an involution, the discriminant of $T$ is a power of two, 
and as the rank of $T$ is two it is either $1,2$ or $4$.
Since $T$ is positive definite and even, we must then have that $T$
has discriminant $4$ and that $T\cong(\ZZ^2,q=2(x^2+y^2))$.

On the other hand, the K3 surface $S$ has the same transcendental lattice 
(this can also be seen 
with a similar argument: in $\Pic(S)$ we have the pull-back of the classes, of a general line in 
$\PP^2$, of the strict transforms of the $4$ exceptional divisors of the triple points and of the 
$15$ exceptional divisors in the second blow up, these $20$ classes are all invariant under 
the covering involution on $S$ and hence, again, $T_S$ is an even lattice
of rank two with discriminant $1,2,4$). Thus the transcendental lattice $T_2$ of 
$S^{[2]}$ is also isomorphic to $T$.

Notice that $\Gamma=\eta^\perp$ in $\tilde{\Gamma}$ and that
the sublattice $\tilde{T}$ of $\tilde{\Gamma}$ spanned by $T$ and $\eta$ is isomorphic to
$(\ZZ^3,2(x^2+y^2+z^2))$. 
A well-known result of Nikulin implies that the embedding of
$\tilde{T}$ in the even unimodular lattice $\tilde{\Gamma}$ is unique up to isometry.
In particular, we may assume it lies in three copies of $U$, the first two in 
$\Lambda_{K3}$, the last spanned by $e_1$ and $e_2$. 
From this one deduces that there is an isometry between the 
lattices $H^2(X_0,\ZZ)$ and $H^2(S^{[2]},\ZZ)$ which restricts to an isometry on 
the transcendental lattices. 
Hence by the Torelli theorem for IHS it follows
that $X_0$ and $S^{[2]}$ are birationally isomorphic.
\end{proof}

\begin{rem}  
We expect that $\mathbb{Z}_2^5 \ltimes \Sigma_5 $ is the group of automorphisms 
of $\overline{S^{[2]}}=X_0$.
Indeed, the group of linear automorphisms of the EPW sextic $Y\subset \PP^5$ is 
$\Sigma_6\rtimes \mathbb{Z}_2^4$
(the permutations of coordinates and the change of an even number of signs).
The linear automorphisms that preserve the twenty ramification planes of $X_0\to Y$
form the group $\mathbb{Z}_2^4 \ltimes \Sigma_5 $ (see Proposition \ref{choices}) and they 
lift to symplectic automorphisms of $X_0$. 
Moreover, the covering involution $X_0\to Y$ is anti-symplectic.
Note that $\mathbb{Z}_2^4 \ltimes \Sigma_5 $ is one of the maximal groups of 
symplectic automorphisms of IHS fourfolds of $K3$ type found in \cite[Thm.~5.1]{HM}. 
Note also that the automorphism group of $S$, and hence of $S^{[2]}$, is infinite. 
\end{rem}

\begin{rem} 
It is natural to consider the map $(S')^{[2]}\dashrightarrow S_5^{[2]}$ induced by the double cover $\rho\colon S'\to S_5$ (see Lemma \ref{rho}).
We can deduce that we have the following diagram:


$$\xymatrix{
  & (S')^{[2]} \ar@{-->}[d]_{\aalpha} \ar@{-->}[r]^{z} & S_5^{[2]} \ar@{-->}[d]_{\pi}& \\
  & \PP^5\supset Y \ar@{-->}[r]^{z'} & \PP^4 & \\
}$$

of rational maps such that the bottom map is the central projection with center 
being a singular point on $Y$ of type $\C^4/(G,i)$. 
By Proposition \ref{tgt cone} the center of projection is of multiplicity $4$ on $Y$. 
It follows that the maps $z'\colon Y \to \PP^4$ and $z$ 
are generically 2:1.  Moreover, the image by $\pi$  of a generic point $\{p,q\}\in S_5^{[2]}$ is the hyperplane of quadrics containing $S_5\subset \PP^5$ and vanishing along the line $<p,q>\subset \PP^5$.
It is easy to see that the double cover $z'$ is ramified along ten hyperplanes 
which are the images of the divisor $\mu(l_i)\subset S^{[2]}$. 
It can be shown that these ten hyperplanes form the configuration of the ten hyperplanes
tangent to the Igusa quartic along quadrics (cf.~ Proposition \ref{tgt cone}). 
It follows that the EPW sextic $Y$
can be constructed as a partial resolution of the double cover of $\PP^4$ 
ramified along the configuration of these ten hyperplanes. 
\end{rem}
\begin{rem}\label{remEx} 
Let us describe how from our picture we obtain a description of 
a symplectic resolution of the singularity $\C^4/G$ considered in \cite{BS}.
We will use the notation from section \ref{section map}. 
We constructed a resolution of singularities 
$$\overline{g}\colon \overline{S^{[2]}}\to E^4/(G,i).$$
The idea now is to look locally at this map around the singular point $\C^4/(G,i)$.
Looking at the Stein factorization $\overline{S^{[2]}}\to E^4/G \to Y$ of 
$\overline{g}$ we see that we need to describe exceptional sets on 
$\overline{S^{[2]}}$ that map to one of the $16$ singular point of type 
$\C^4/G$ on $E^4/G$. 

We shall use the geometric definition of the map $g$.
Recall that the nodal $K3$ surface $S'\subset \cone\subset \PP^6$ is contained in the cone 
$\cone$ with vertex $P$.
A general line in this cone passing through $P$ cuts $S'$ in two points.
On the other hand $\cone$ is contained in a four dimensional system $Q$ of quadrics.
Denote the set of points in $(S')^{[2]}$ that correspond to lines cutting $S'$ 
in two points and contained 
in $\cone$ by $A\subset (S')^{[2]}$.
From the definition of the map $\aalpha$, the set $A\subset (S')^{[2]}$ maps 
to the point $(1:\ldots:1)\in Y$ corresponding to $Q$ 
(cf.\ the end of the proof of Proposition \ref{equationY}).

Note that for each line $\phi(l_1)\subset S'$ determined by $l_1\subset S$ we have a plane 
$P_{l_1}$ in $\PP^5$ spanned by $P$ 
and $\phi(l_1)$.
 The rational curve $l_1$ cuts three exceptional curves $e_1$, $e_2$ and $e_3$ in points.
All the lines on the plane $P_{l_1}$ are tangent to $S'$ and they determine a surface 
$E_{l_1}$ in $S^{[2]}$.
The surfaces $e_1^{[2]}$, $e_2^{[2]}$ and $e_3^{[2]}$ intersect $E_{l_1}$ in points 
and they also intersect 
two of the indeterminacy loci $E_{13}$, $E_{12}$ and $E_{11}$ of type (3) along lines. 
Moreover, $E_{l_1}$ is isomorphic to $\mathbb{F}_4$.  
It follows that the strict transform of $E_{l_1}$ on 
$\overline{S}^{[2]}$ is isomorphic to $\mathbb{P}^2$. 

The surfaces $l^{[2]}_1$ and $E_{l_1}$ intersect on $S^{[2]}$ along a
conic curve $c$.  The proper transform of this conic curve $c$ before
the last flop is a line contained in the proper transform of
$l_1^{[2]}$ (the Cremona transformation of $l_1^{[2]}$ with center at
three points on $c$ transforms $c$ to a line).  Before the last flop
the proper transform of the surface $E_{l_1}$ will be $\mathbb{F}_1$
(i.e. $\mathbb{F}_4$ transformed by three elementary transformations).
After the last flop the proper transform of the above line will be
contracted such that the proper transform of $E_{l_1}$ will be a
plane.

Finally, consider also the  del Pezzo surface $A$ of degree 5 contained in $S^{[2]}$, 
corresponding to the rays of the cone $\cone$. 
The strict transform on $\overline{S^{[2]}}$ it is still a del Pezzo surface of degree 5. 
These eleven surfaces form the exceptional set of the resolution of the singular point
$\C^4/G$ (cf.~\cite{DW}).
\end{rem}


\begin{remark}\label{combinatorics}
Some result of the present paper can be neatly illustrated with the
combinatorics related to the Petersen graph.

Firstly, let us recall that the graph describes the incidence of
$(-1)$-curves on ${\mathbb P}^2_4$ with 10 vertices standing for the
curves and edges for their intersection. Using this notation we can
encode five conic pencils on $S_5={\mathbb P}^2_4$: each pencil
contains three reducible fibers which can be represented by three
edges in this graph (two lines plus their point of
intersection). Three double edges in the graph below stand for such a
pencil. Thus the 15 edges of the graph are divided into 5 classes;
every edge in each class shares no common adjacent edge with another
edge in the same class. The five conic pencils on $S_5$ give five
distinguished elliptic pencils on the Vinberg's K3 surface $S$, see section
\ref{section k3s}.

It is well known, see the first part of Lemma \ref{auts5}, that
$S_5\subset{\mathbb P}^5$ admits an action of the permutation group
$\Sigma_5$ which yields an action of $\Sigma_5$ on the set of
$(-1)$-curves on $S_5$. In fact, we can identify the $(-1)$-curves
with transpositions and the action of $\Sigma_5$ is then by
conjugation. This is depicted in the diagram below by assigning to
each vertex of the Petersen graph a pair from the set $\{a,\dots,e\}$.

There are exactly two non-conjugate embeddings $\Sigma_5
\hookrightarrow \Sigma_6$: apart of the standard one (coming from the
embedding $\{0,\dots,4\}\hookrightarrow\{0,\dots,5\}$) there is an
embedding $\beta$ described in the second part of Lemma
\ref{auts5}. The embedding $\beta$ assigns to a pair in 5-element set
(a transposition in $\Sigma_5$) a partition of a 6-element set into 3
pairs. 

We label the edges of the Petersen graph by pairs in $\{0,\dots,5\}$
so that the three edges stemming from a given vertex give the
respective partition; for example: $(ab)\mapsto (03)(14)(25)$,
$(bc)\mapsto(01)(24)(35)$, $(cd)\mapsto (05)(14)(23)$, $(de)\mapsto
(01)(25)(34)$, compare it with \ref{auts5}. This way, out of 15
partitions of the set $\{0,\dots,5\}$, ten can be associated to
vertices of the Petersen graph. As the result, the remaining 5
partitions come from the 5 triples of edges which are associated to
conic pencils on $S_5$ or distinguished elliptic pencils on $S$. In
our diagram they are the following: $(01)(23)(45)$, $(02)(14)(35)$,
$(03)(15)(24)$, $(04)(13)(25)$, $(05)(12)(34)$. These are exactly the
partitions which occur in Proposition \ref{20inc}; there they come to
define the 20 incident planes: each partition defines $3+1$ planes,
depending on the number of $\pm$ signs. Similarly, the divisors $B_i$
in Proposition \ref{singZ} can be divided into 5 classes related to 5
distinguished elliptic fibrations of $S$, each of the classes
containing $3+1$ divisors.

$$\begin{xy}<120pt,0pt>:
(0,1)*=<15pt,15pt>[o]{ab}="01",
(0.588,-0.809)*=<15pt,15pt>[o]{ac}="02",
(0.475,0.13)*=<15pt,15pt>[o]{bc}="12", 
(-0.588,-0.809)*=<15pt,15pt>[o]{be}="14",
(-0.475,0.13)*=<15pt,15pt>[o]{ae}="04",
(-0.95,0.27)*=<15pt,15pt>[o]{cd}="23",
(0.294,-0.404)*=<15pt,15pt>[o]{bd}="13",
(0.95,0.27)*=<15pt,15pt>[o]{de}="34",
(-0.294,-0.404)*=<15pt,15pt>[o]{ad}="03",
(0,0.5)*=<15pt,15pt>[o]{ce}="24",
"01";"24" **@{-}, (0.1,0.7)*={(03)},
"01";"23" **@{-}, (-0.5,0.73)*={(14)},
"01";"34" **@{-}, (0.5,0.73)*={(25)},
"02";"34" **@{-}, (0.9,-0.25)*={(34)},
"02";"13" **@{-}, (0.33,-0.6)*={(02)},
"02";"14" **@{-}, (0,-0.9)*={(15)},
"03";"12" **@{-}, (0.4,-0.03)*={(35)},
"03";"14" **@{-}, (-0.33,-0.6)*={(04)},
"03";"24" **@{-}, (-0.16,0.3)*={(12)},
"04";"12" **@{-}, (0,0.08)*={(24)},
"04";"13" **@{-}, (-0.4,-0.03)*={(13)},
"04";"23" **@{-}, (-0.7,0.13)*={(05)},
"12";"34" **@{=}, (0.7,0.13)*={(01)},
"13";"24" **@{=}, (0.16,0.31)*={(45)},
"14";"23" **@{=}, (-0.9,-0.25)*={(23)},
\end{xy}$$

\end{remark}

\bibliography{ihsbib} \bibliographystyle{alpha}

\end{document}